\documentclass[11pt,letterpaper]{amsart}

\usepackage[T1]{fontenc}
\usepackage{lmodern}
\usepackage[expansion=false]{microtype}
\usepackage{mathtools}
\usepackage{amssymb}
\usepackage{mathrsfs}
\usepackage{enumitem}
\usepackage[margin=1.08in]{geometry}
\usepackage[
  colorlinks=true,
  linkcolor=blue,
  citecolor=red,
  urlcolor=cyan,
  pagebackref=true
]{hyperref}
\usepackage[capitalise,noabbrev]{cleveref}

\numberwithin{equation}{section}

\newtheorem{thm}{Theorem}[section]
\newtheorem{prop}[thm]{Proposition}
\newtheorem{lem}[thm]{Lemma}
\newtheorem{cor}[thm]{Corollary}
\theoremstyle{definition}
\newtheorem{defn}[thm]{Definition}
\theoremstyle{remark}
\newtheorem{remark}[thm]{Remark}

\crefname{thm}{theorem}{theorems}
\Crefname{thm}{Theorem}{Theorems}
\crefname{prop}{proposition}{propositions}
\Crefname{prop}{Proposition}{Propositions}
\crefname{lem}{lemma}{lemmas}
\Crefname{lem}{Lemma}{Lemmas}
\crefname{cor}{corollary}{corollaries}
\Crefname{cor}{Corollary}{Corollaries}
\crefname{defn}{definition}{definitions}
\Crefname{defn}{Definition}{Definitions}
\crefname{remark}{remark}{remarks}
\Crefname{remark}{Remark}{Remarks}

\newcommand{\C}{\mathbf C}
\newcommand{\Pj}{\mathbf P}
\newcommand{\D}{\mathbb D}
\newcommand{\FS}{\mathrm{FS}}
\newcommand{\PSH}{\operatorname{PSH}}
\newcommand{\MA}{\operatorname{MA}}
\newcommand{\ord}{\operatorname{ord}}
\newcommand{\norm}[1]{\lVert#1\rVert}
\newcommand{\ddc}{dd^c}

\hypersetup{
  pdftitle  = {Prescribed Lelong Numbers for One-Pole Green Functions on Complex Projective Space},
  pdfauthor = {Xiangsen Qin},
  pdfsubject = {One-Pole Green functions and Lelong numbers},
  pdfkeywords = {Green functions, elong numbers, Seshadri constants}
}

\begin{document}

\title[Prescribed Lelong numbers on projective space]
{Prescribed Lelong Numbers for One-Pole Green Functions on Complex Projective Space}

\author[X. Qin]{Xiangsen Qin}
\address{Xiangsen Qin: Chern Institute of Mathematics and LPMC, Nankai University \\
Tianjin 300071, China}
\email{qinxiangsen@nankai.edu.cn}

\begin{abstract}
Let $\omega_{\mathrm{FS}}$ be the normalized Fubini--Study form on
$\mathbf P^n$, with $n\geq2$.  We prove that the one-pole Lelong-number
range in $DMA(\mathbf P^n,\omega_{\mathrm{FS}})$ is exactly $[0,1]$:
for every $\lambda$ in this interval there is a Green function with a
single pole, Monge--Amp\`ere measure equal to the Dirac mass at that
pole, and Lelong number $\lambda$.  This answers Question~9 in the
survey of Dinew--Guedj--Zeriahi, where the problem is attributed to
Coman and Guedj.  The construction adapts Li and Xia's local
zero-Lelong-number
scheme through variable degrees and homogeneous finite stages, while
also establishing the exact Lelong number and membership in the global
$DMA$ class.  We then study the relation between the one-pole range
$\mathcal R_\alpha(x)$ and the Seshadri interval
$[0,\varepsilon(\alpha,x)]$.  An application of Koike's equivalence
theorem gives a point on a degree-one del Pezzo surface where the
Seshadri endpoint is not attained.  Conversely, a finite-pullback
criterion and an explicit finite morphism show that every ample rational
class on a product of projective spaces realizes its full Seshadri
interval at every point.
\end{abstract}


\maketitle
\tableofcontents

\section{Introduction}

Let $\omega_{\FS}$ denote the Fubini--Study form on complex projective
space.  On every $\Pj^m$ occurring below it is normalized by
\[
  \int_{\Pj^m}\omega_{\FS}^m=1.
\]
An $\omega_{\FS}$-plurisubharmonic function $G$ in the global domain
$DMA(\Pj^n,\omega_{\FS})$ is called a one-pole Green function at
$a\in\Pj^n$ if
\[
  (\omega_{\FS}+\ddc G)^n=\delta_a,
  \qquad G^{-1}(-\infty)=\{a\}.
\]
Locally, products of bounded psh potentials and their monotone continuity
are understood in the sense of Bedford--Taylor \cite{BT82}.  For the
unbounded global potentials considered here, the left-hand side is the
classical measure in the global decreasing-continuity domain
$DMA(\Pj^n,\omega_{\FS})$ introduced by Coman--Guedj--Zeriahi
\cite{CGZ08}.  It is
not the non-pluripolar product, which would discard the mass carried by the
pole.

Coman and Guedj constructed one-pole Green functions on $\Pj^2$ with
every rational Lelong number in $(0,1]$; see
\cite[Sections~2.2.1 and 2.2.4]{CG09}.  The first explicit published
formulation we have located appears in
\cite[Section~1.3, Question~9, pp.~907--908]{DGZ16}, where it is
attributed to Coman and Guedj.  The earlier paper \cite{CG09} provides
the antecedent rational constructions but contains no separately stated
irrational-value question.

Li and Xia later constructed, in every complex dimension at least two,
a local plurisubharmonic function with an isolated pole, unit residual
Monge--Amp\`ere mass, and zero Lelong number
\cite[Sections~3.1--3.5]{LX26}.  Their construction supplies the
finite-stage iteration, moving truncations, persistence of inactive
cutoffs, real-ray witness, and max-product suspension used below.  In
their fixed-degree quadratic scheme, however, the normalized
finite-stage Lelong numbers tend to zero; this does not prescribe an
arbitrary positive value.

The construction in \Cref{sec:finite-stages,sec:truncations,sec:limiting-green,sec:suspension} follows that architecture with the
changes required by the global prescribed-value problem.  Variable
pairs $(k_j,D_j)$ control the normalized vanishing orders and make them
converge to a prescribed $\lambda\in[0,1)$.  Homogeneous
projectivization turns every finite stage into a global object on
$\Pj^2$.  A Schwarz estimate with multiplicity and a real-ray sequence
give the two inequalities needed to identify the Lelong number exactly.
Finally, the decreasing limit and its suspension must be shown to belong
to the global $DMA$ class, rather than only to a local Monge--Amp\`ere
domain.  These are substantive modifications of the Li--Xia scheme for
the present global problem, not a separate construction architecture.

We first determine the full range on projective space.

\begin{thm}\label{thm:main}
Let $n\geq2$, let $a=[1:0:\cdots:0]\in\Pj^n$, and let
$\lambda\in[0,1]$.  There exists
\[
 G_{n,\lambda}\in
 \PSH(\Pj^n,\omega_{\FS})\cap DMA(\Pj^n,\omega_{\FS})
\]
whose pole set is $\{a\}$ and such that
\begin{equation}\label{eq:main-equation}
 (\omega_{\FS}+\ddc G_{n,\lambda})^n=\delta_a,
 \qquad
 \nu(G_{n,\lambda},a)=\lambda.
\end{equation}
Moreover, $G_{n,\lambda}$ is continuous on $\Pj^n\setminus\{a\}$, and
$\exp(G_{n,\lambda})$, extended by zero at $a$, is continuous on $\Pj^n$.
\end{thm}

Every positive closed current in the normalized Fubini--Study class has
Lelong number at most one at every point
\cite[Proposition~2.1]{CG09}.  Thus \Cref{thm:main} gives the complete
range.

\begin{cor}\label{cor:range}
For each $n\geq2$ and each fixed $a\in\Pj^n$, the set of numbers
\[
 \nu(G,a),
 \quad
 G\in DMA(\Pj^n,\omega_{\FS}),
 \quad
 (\omega_{\FS}+\ddc G)^n=\delta_a,
 \quad
 G\in L^\infty_{\mathrm{loc}}(\Pj^n\setminus\{a\}),
\]
is exactly $[0,1]$.
\end{cor}

Consequently, \Cref{cor:range} answers \cite[Question~9]{DGZ16} and
determines the complete one-pole Lelong-number range on projective space.

To formulate the corresponding problem on a general compact K\"ahler
manifold, let $X$ be a compact K\"ahler $n$-fold, let $\alpha$ be a
K\"ahler class with $\alpha^n=1$, and fix $x\in X$.  Choose a K\"ahler
form $\vartheta\in\alpha$ and define the one-pole Lelong range by
\begin{equation}\label{eq:projective-range}
 \mathcal R_\alpha(x)=
 \left\{\nu(G,x):
 \begin{array}{l}
 G\in\PSH(X,\vartheta)\cap DMA(X,\vartheta),\\
 G\in L^\infty_{\mathrm{loc}}(X\setminus\{x\}),\quad
 (\vartheta+\ddc G)^n=\delta_x
 \end{array}\right\}.
\end{equation}
This set depends only on the cohomology class \(\alpha\), not on the choice of representative \(\vartheta\).

We also use the Seshadri constant of a K\"ahler class $\alpha$ at $x$.
If
\[
 \pi_x:\operatorname{Bl}_xX\longrightarrow X
\]
is the blow-up and $E_x$ is its exceptional divisor, then
\begin{equation}\label{eq:seshadri-intro}
 \varepsilon(\alpha,x)=
 \sup\{t\geq0:\pi_x^*\alpha-t\{E_x\}\ \text{is nef}\}.
\end{equation}

The Seshadri constant is relevant here because it bounds the strength of
an isolated singularity in the class $\alpha$.  More precisely,
\Cref{prop:seshadri-obstruction} gives
\[
 \mathcal R_\alpha(x)\subset[0,\varepsilon(\alpha,x)].
\]
This inclusion raises two distinct questions: whether every value below
the upper bound is realized, and whether the endpoint itself is attained.

The three main results give complementary answers.  For $n\geq2$,
\Cref{thm:main} constructs, for the normalized Fubini--Study class on
$\Pj^n$, a one-pole Green function for every $\lambda\in[0,1]$; each
potential is continuous away from the pole, and its exponential extends
continuously by zero at the pole.  Consequently, \Cref{cor:range} gives
\[
 \mathcal R_{c_1(\mathcal O_{\Pj^n}(1))}(a)
 =[0,1]
 =[0,\varepsilon(c_1(\mathcal O_{\Pj^n}(1)),a)].
\]
At the opposite extreme, the following \Cref{thm:koike-obstruction} produces a point
$x$ on a degree-one del Pezzo surface and a volume-normalized class
$\alpha=c_1(-K_X)$ for which
\[
 \varepsilon(\alpha,x)=1,
 \qquad 1\notin\mathcal R_\alpha(x).
\]
Finally, let $X=\prod_{j=1}^k\Pj^{n_j}$, where $n_j\geq1$ and
$N=\sum_j n_j\geq2$, let $H_j$ be the pullback of the hyperplane class
from the $j$th factor, and set $A=\sum_j a_jH_j$ with
$a_j\in\mathbf Q_{>0}$ and $\alpha=A/(A^N)^{1/N}$.  Then
\Cref{thm:projective-space-products} proves, at every $x\in X$, the
exact formula
\[
 \mathcal R_\alpha(x)
 =[0,\varepsilon(\alpha,x)]
 =\left[0,\frac{\min_j a_j}{(A^N)^{1/N}}\right].
\]
Thus the del Pezzo example isolates a genuine endpoint obstruction,
whereas products of projective spaces realize the entire Seshadri
interval.

Coman and Guedj constructed endpoint Green functions at several special
points of degree-one del Pezzo surfaces and left the arbitrary-point
case open \cite[Section~4.3]{CG09}.  The next result applies Koike's
equivalence theorem to a carefully chosen nodal anticanonical curve.
The positive-current uniqueness mechanism is Koike's; the contribution
here is to convert that uniqueness, through a blow-up and Siu
decomposition, into a concrete obstruction to endpoint attainment.

\begin{thm}
\label{thm:koike-obstruction}
There exist a degree-one del Pezzo surface $X$, a point $x\in X$, and the
volume-normalized K\"ahler class $\alpha=c_1(-K_X)$ such that
\[
 \alpha^2=1,
 \qquad \varepsilon(\alpha,x)=1,
\]
but no positive closed current with cohomology class $\alpha$ and locally
bounded potentials on $X\setminus\{x\}$ has Lelong number one at $x$.
In particular, no one-pole Green function in $\alpha$ realizes the
Seshadri endpoint.
\end{thm}

For rational $b\geq1$, Coman and Guedj constructed, in their
normalization on $\Pj^1\times\Pj^1$, one-pole Green functions realizing
every positive rational Lelong number in the corresponding admissible
interval \cite[Example~3.5]{CG09}.  Writing $\operatorname{pr}_i$ for
the two factor projections, their class is
\[
 c_1\!\left(\operatorname{pr}_1^*\mathcal O_{\Pj^1}(1)\right)
 +b\,c_1\!\left(\operatorname{pr}_2^*\mathcal O_{\Pj^1}(1)\right),
\]
whose volume is $2b$; unit-volume normalization multiplies the Lelong
numbers by $(2b)^{-1/2}$.  The next theorem shows that, for every ample rational class on a product of projective spaces, every value between zero and the Seshadri constant occurs as the Lelong number of a one-pole Green function. Thus its content goes beyond the computation of the Seshadri constant alone.

\begin{thm}
\label{thm:projective-space-products}
Let
\[
 X=\prod_{j=1}^k\Pj^{n_j},
 \qquad n_j\geq1,
 \qquad N=\sum_{j=1}^k n_j\geq2,
\]
and set
\[
 H_j=c_1\bigl(\operatorname{pr}_j^*\mathcal O_{\Pj^{n_j}}(1)\bigr),
 \qquad
 A=\sum_{j=1}^k a_jH_j,\qquad a_j\in\mathbf Q_{>0},
 \qquad
 \alpha=\frac{A}{(A^N)^{1/N}}.
\]
Here $\operatorname{pr}_j$ is the projection to the $j$th factor, and
$A^N$ denotes the top self-intersection number of $A$.
Then, for every $x\in X$,
\[
 \mathcal R_\alpha(x)
 =[0,\varepsilon(\alpha,x)]
 =\left[0,\frac{\min_j a_j}{(A^N)^{1/N}}\right].
\]
\end{thm}

The remainder of this paper is organized as follows. \Cref{sec:preliminaries} fixes notation and recalls the basic facts used throughout the paper.
\Cref{sec:finite-stages,sec:truncations,sec:limiting-green} construct the
prescribed one-pole Green functions on $\Pj^2$ and verify their
Monge--Amp\`ere measure and Lelong number.  \Cref{sec:suspension} proves
the max-product suspension and completes the proofs of
\Cref{thm:main} and Corollary \ref{cor:range}.
\Cref{subsec:seshadri-bound,subsec:nonattained-endpoint} establish the
general Seshadri upper bound and the degree-one del Pezzo obstruction.
\Cref{subsec:finite-pullbacks} proves the generic-rotation and
finite-pullback results.
\Cref{subsec:products,subsec:further-consequence} apply them to
products of projective spaces and to the final multipole consequence.

\section{Preliminaries}\label{sec:preliminaries}

We use
\begin{equation}\label{eq:ddc}
 d^c=\frac{1}{2\pi i}(\partial-\bar\partial),
 \qquad
 \ddc=\frac{i}{\pi}\partial\bar\partial,
\end{equation}
so that $\ddc\log|z|=\delta_0$ in one complex dimension.  Let
$\pi_n:\C^{n+1}\setminus\{0\}\to\Pj^n$ be the standard projection and
write $\norm{X}_2$ and $\norm{X}_\infty$ for the Euclidean and maximum
norms on $\C^{n+1}$.  Our normalization of the Fubini--Study form is
\begin{equation}\label{eq:FS}
 \pi_n^*\omega_{\FS}=\ddc\log\norm{X}_2.
\end{equation}
We write $\D=\{\tau\in\C:|\tau|<1\}$ for the unit disc, so that
$\D^n$ is the unit polydisc in $\C^n$.

For positive Radon measures $\mu_j$ and $\mu$ on a compact space $M$,
we write $\mu_j\rightharpoonup\mu$ if
$\int_M\chi\,d\mu_j\to\int_M\chi\,d\mu$ for every
$\chi\in C^0(M)$.  On a locally compact space $\Omega$,
$\mu_j\rightharpoonup\mu$ locally means the same for every
$\chi\in C_c^0(\Omega)$.

Let $M$ be a compact complex manifold and let $\vartheta$ be a smooth
closed real $(1,1)$-form on $M$.  An integrable upper semicontinuous
function $u$ is called $\vartheta$-plurisubharmonic, or
$\vartheta$-psh, if $\vartheta+\ddc u$ is a positive current.  The set of
these functions is denoted by $\PSH(M,\vartheta)$.  If
$\ddc\psi=\vartheta$ locally, then the Lelong number $\nu(u,x)$ means the
Lelong number of the psh function $u+\psi$; it is independent of the
choice of $\psi$.
For a closed real $(1,1)$-form $\vartheta$ or a divisor $D$, braces such as
$\{\vartheta\}$ and $\{D\}$ denote the corresponding real cohomology
classes.

\begin{defn}\label{def:DMA}
Let $M$ be a compact K\"ahler manifold of complex dimension $m$, and let
$\vartheta$ be a K\"ahler form on $M$.  A function
$u\in\PSH(M,\vartheta)$ belongs to
$DMA(M,\vartheta)$ if there is a positive Radon measure
$\MA_{\vartheta}(u)$ such that, for every sequence of bounded
$\vartheta$-psh functions $u_j\downarrow u$,
\[
 (\vartheta+\ddc u_j)^m
 \rightharpoonup \MA_{\vartheta}(u).
\]
When this holds we write
$\MA_{\vartheta}(u)=(\vartheta+\ddc u)^m$.
\end{defn}

This is the global decreasing-continuity domain introduced in
\cite{CGZ08}.  The products of bounded approximants in
\Cref{def:DMA} are the local Bedford--Taylor products; the requirement
that their limits be independent of every global decreasing approximation
is the additional global condition in \cite{CGZ08}.  We use the following
consequence of the local theory.

For a domain $\Omega\subset\C^N$, let $\mathcal D(\Omega)$ denote the
local domain of definition of the complex Monge--Amp\`ere operator: a
function $h\in\PSH(\Omega)$ belongs to $\mathcal D(\Omega)$ when, for
every sequence of locally bounded psh functions decreasing to $h$, the
top Monge--Amp\`ere measures converge locally to a positive Radon measure
independent of the sequence.  We write
$DMA_{\mathrm{loc}}(M,\vartheta)$ in the sense of
\cite[Definition~3.1]{CGZ08}: after adding a smooth local potential of
$\vartheta$, the function belongs to $\mathcal D$ in every coordinate
chart.  This local condition is distinct from the global class in
\Cref{def:DMA}.  Theorem~3.2 of \cite{CGZ08}, together with
Definitions~1.5 and~1.7 there, yields the consequence
\begin{equation}\label{eq:DMA-chain}
 DMA_{\mathrm{loc}}(M,\vartheta)
 \subset DMA(M,\vartheta).
\end{equation}
\begin{lem}\label{lem:isolated-DMA}
Let $M$ be a compact K\"ahler manifold, let $\vartheta$ be a K\"ahler
form on $M$, and let $x\in M$.  If
$u\in\PSH(M,\vartheta)$ is locally bounded on $M\setminus\{x\}$, then
$u\in DMA(M,\vartheta)$.
Furthermore, its global Monge--Amp\`ere measure is the weak limit along
every decreasing bounded $\vartheta$-psh approximation and agrees on
$M\setminus\{x\}$ with the local Bedford--Taylor measure.
The same conclusion holds on a smooth projective manifold when the
one-point set is replaced by any finite set.
\end{lem}

\begin{proof}
Let $u_j\downarrow u$ be bounded $\vartheta$-psh functions.  Every weak
cluster limit of $(\vartheta+\ddc u_j)^m$ agrees on
$M\setminus\{x\}$ with the local Bedford--Taylor measure of $u$, by
local Bedford--Taylor monotone convergence.  Thus any two cluster
limits, including limits from different decreasing approximations,
differ by $c\delta_x$.  For every $j$, Stokes' theorem for bounded
Bedford--Taylor products gives
\[
 \int_M(\vartheta+\ddc u_j)^m=\int_M\vartheta^m.
\]
All cluster limits therefore have the same total mass, so $c=0$.  Weak
compactness of positive measures of fixed mass on the compact space
$M$ now gives convergence of the whole sequence and independence of
the approximation.

For a finite set on a projective manifold, a general member of a
sufficiently high very ample linear system is a smooth ample divisor $D$
disjoint from that set.  The function $u$ is bounded near $D$, so
\cite[Proposition~4.6]{CGZ08} gives
$u\in DMA_{\mathrm{loc}}(M,\vartheta)$.  The inclusion
\eqref{eq:DMA-chain} then places $u$ in $DMA(M,\vartheta)$, and the
asserted convergence follows from \Cref{def:DMA}.
\end{proof}

For a psh function $u$ near the origin in $\C^n$, we use the Lelong-number
normalization
\begin{equation}\label{eq:lelong-liminf}
 \nu(u,0)=\liminf_{z\to0}\frac{u(z)}{\log\norm{z}_\infty}.
\end{equation}
In particular, if $F=(F_1,\ldots,F_N)$ is a holomorphic germ whose
components vanish at zero, then
\begin{equation}\label{eq:ideal-lelong}
 \nu(\log\norm{F}_\infty,0)=
 \min_{1\leq r\leq N}\ord_0F_r.
\end{equation}

\section{Rational stages on the projective plane}\label{sec:finite-stages}

We first work on $\Pj^2$.  Higher dimensions will follow from the
max-product construction in \cref{sec:suspension}.  The finite polynomial
iteration is adapted from \cite[Section~3.1]{LX26}; here the varying pairs
$(k_j,D_j)$ prescribe the normalized vanishing orders, and homogeneity
makes every stage global on projective space.

Fix $0\leq\lambda<1$.  Choose positive rational numbers
\begin{equation}\label{eq:rational-sequence}
 1=\lambda_0>\lambda_1>\lambda_2>\cdots>\lambda,
 \qquad \lambda_j\longrightarrow\lambda.
\end{equation}
Write
\[
 \frac{\lambda_j}{\lambda_{j-1}}=\frac{k_j}{D_j},
 \qquad 1\leq k_j<D_j,
\]
with integers $k_j,D_j$, and set
\[
 d_0=1,
 \qquad d_j=\prod_{s=1}^jD_s.
\]
The parameters $0<\varepsilon_j\leq1/4$ are chosen recursively in
\cref{sec:truncations}; every finite-stage assertion below holds for
arbitrary previously chosen $\varepsilon_1,\ldots,\varepsilon_j$.

Use homogeneous coordinates $[T:Z:W]$ on $\Pj^2$.  Define a
holomorphic endomorphism of algebraic degree $D_j$ by
\begin{equation}\label{eq:Phi}
 \Phi_j[T:Z:W]
 =\left[
 T^{D_j}:\frac12W^{D_j}:
 \frac12Z^{D_j}+\varepsilon_jW^{k_j}T^{D_j-k_j}
 \right].
\end{equation}
Its coordinates have no common zero, and
$\Phi_j^{-1}(a)=\{a\}$ for $a=[1:0:0]$.

Let
\begin{equation}\label{eq:Psi}
 \Psi_j=\Phi_j\circ\cdots\circ\Phi_1
 =[T^{d_j}:P_j:Q_j].
\end{equation}
For the initial stage set
\[
 P_0=Z,
 \qquad Q_0=W.
\]
The affine germ of $\Phi_j$ at $a$ is
\begin{equation}\label{eq:affine-map}
 \varphi_j(z,w)=
 \left(\frac12w^{D_j},
 \frac12z^{D_j}+\varepsilon_jw^{k_j}\right).
\end{equation}

\begin{lem}\label{lem:degree-orders}
The non-leading coordinates in \eqref{eq:Psi} have common zero set
$\{a\}$ and satisfy
\begin{equation}\label{eq:orders}
 \ord_0P_j(1,\cdot,\cdot)=\lambda_{j-1}d_j,
 \qquad
 \ord_0Q_j(1,\cdot,\cdot)=\lambda_jd_j.
\end{equation}
Consequently, their minimum vanishing order is
\begin{equation}\label{eq:min-order}
 o_j=\lambda_jd_j.
\end{equation}
\end{lem}

\begin{proof}
Since $\Phi_j^{-1}(a)=\{a\}$ at every stage, the non-leading coordinates
of $\Psi_j$ vanish simultaneously only at $a$.

Put
$\alpha_j=\ord_0P_j(1,\cdot,\cdot)$ and
$\beta_j=\ord_0Q_j(1,\cdot,\cdot)$.  From \eqref{eq:Phi}--\eqref{eq:Psi},
\[
 \alpha_j=D_j\beta_{j-1},
 \qquad
 \beta_j=\min\{D_j\alpha_{j-1},k_j\beta_{j-1}\}.
\]
We prove the formulas by induction, starting from
$\alpha_0=\beta_0=1$.  If they hold at stage $j-1$, then
\[
 \alpha_j=D_j\beta_{j-1}=\lambda_{j-1}d_j,
 \qquad
 k_j\beta_{j-1}=\lambda_jd_j.
\]
For $j\geq2$, the other term in the minimum is
$D_j\alpha_{j-1}=\lambda_{j-2}d_j>\lambda_jd_j$; for $j=1$, the two
terms are $D_1$ and $k_1$, with $k_1<D_1$.  Thus
$\beta_j=\lambda_jd_j$.  This proves \eqref{eq:orders} and
\eqref{eq:min-order}.
\end{proof}

For $X=(T,Z,W)\in\C^3$, put
\[
 \mathfrak r_j(X)=\max\{|P_j(X)|,|Q_j(X)|\}
\]
and define
\begin{equation}\label{eq:qj}
 q_j([X])=\frac{1}{d_j}\log \mathfrak r_j(X)-\log\norm{X}_2.
\end{equation}
The homogeneity of the two terms makes $q_j$ well defined on $\Pj^2$.

\begin{prop}\label{prop:finite-green}
For every $j\geq1$,
$q_j\in\PSH(\Pj^2,\omega_{\FS})$ and
$q_j$ is continuous on $\Pj^2\setminus\{a\}$.  Moreover,
\begin{equation}\label{eq:finite-stage-local}
 \nu(q_j,a)=\lambda_j,
 \qquad
 (\omega_{\FS}+\ddc q_j)^2=0
 \quad\text{on }\Pj^2\setminus\{a\}.
\end{equation}
\end{prop}

\begin{proof}
By \eqref{eq:FS},
\[
 \pi_2^*(\omega_{\FS}+\ddc q_j)
 =d_j^{-1}\ddc\log \mathfrak r_j\geq0.
\]
Thus $q_j$ is $\omega_{\FS}$-psh.  Its only pole is the common zero set
$\{a\}$, and its defining formula is continuous off that point.  Moreover,
\eqref{eq:ideal-lelong}, \eqref{eq:min-order}, and \eqref{eq:qj} give
$\nu(q_j,a)=\lambda_j$.

On $\Pj^2\setminus\{a\}$, let
$f_j=[P_j:Q_j]$ and define on $\Pj^1$
\[
 h([\xi])=\log\norm{\xi}_\infty-\log\norm{\xi}_2.
\]
The function $h$ is bounded and $\omega_{\FS}$-psh, with
\[
 \omega_{\FS}+\ddc h=\ddc\log\norm{\xi}_\infty\geq0.
\]
On $\Pj^2\setminus\{a\}$,
\[
 \omega_{\FS}+\ddc q_j
 =d_j^{-1}f_j^*(\omega_{\FS}+\ddc h).
\]
Indeed, after pullback to the homogeneous cone, both sides equal
$d_j^{-1}\ddc\log\mathfrak r_j$.
The pullback of this nonsmooth current is defined using bounded local
potentials.  On a target coordinate disc, write
$\omega_{\FS}+\ddc h=\ddc v$ with $v$ bounded and subharmonic.  If
$v_\ell\downarrow v$ are smooth subharmonic approximants on a smaller
disc, the corresponding source potentials are
$d_j^{-1}v_\ell\circ f_j$.  Since the target has complex dimension one,
\[
 \bigl(\ddc(d_j^{-1}v_\ell\circ f_j)\bigr)^2
 =d_j^{-2}f_j^*\bigl((\ddc v_\ell)^2\bigr)=0.
\]
Bedford--Taylor monotone continuity therefore gives
$(\omega_{\FS}+\ddc q_j)^2=0$ away from $a$.
\end{proof}

\section{Moving truncations and their limit}\label{sec:truncations}

The functions $q_j$ have the required finite-stage Lelong numbers, but
they are not monotone in $j$.  We truncate them at levels tending to
$-\infty$ without changing them outside shrinking neighborhoods of the
pole.  This moving-truncation mechanism is adapted from
\cite[Sections~3.2--3.4]{LX26}; the cutoff depths below are chosen
inductively because both the degrees and the prescribed vanishing orders
now vary with $j$.

Define the bounded $\omega_{\FS}$-psh function
\[
 \rho([X])=\log\norm{X}_\infty-\log\norm{X}_2
\]
and the scale-invariant function
\[
 p_j([X])=\frac{1}{d_j}\log \mathfrak r_j(X)-\log\norm{X}_\infty,
 \qquad q_j=\rho+p_j.
\]

Choose numbers
\[
 0<\eta_{j+1}<\eta_j<\operatorname{diam}(\Pj^2,\omega_{\FS}),
 \qquad \eta_j\longrightarrow0,
\]
and let
\[
 U_j=B_{\omega_{\FS}}(a,\eta_j)
\]
be the corresponding Fubini--Study geodesic balls.  Thus $U_j$ is a
neighborhood basis at $a$: every compact subset of
$\Pj^2\setminus\{a\}$ is disjoint from all sufficiently late $U_j$.

We now specify the recursion.  The data
$\lambda_j,k_j,D_j,d_j$ and $U_j$ have already been fixed.  Set $L_0=0$
and choose $0<\varepsilon_1\leq1/4$.  At stage $j$, once
$\varepsilon_1,\ldots,\varepsilon_j$ are known,
\eqref{eq:Phi}--\eqref{eq:qj} determine
$\Phi_j,\Psi_j,P_j,Q_j,\mathfrak r_j,q_j$, and $p_j$.  Along the positive real ray
\[
 (z,w)=(0,t),
 \qquad 0<t<1,
\]
all coefficients of both components of the iterated affine map are
nonnegative.  This holds at the first stage and is preserved by
\eqref{eq:affine-map}.  If the estimate below holds at stage $j$, then
the perturbation term at the next stage gives
\[
 \varepsilon_{j+1}Q_j(1,0,t)^{k_{j+1}}
 \geq \varepsilon_{j+1}\gamma_j^{k_{j+1}}
 t^{k_{j+1}o_j}.
\]
Here $k_{j+1}o_j=o_{j+1}$ follows from
$\lambda_{j+1}/\lambda_j=k_{j+1}/D_{j+1}$ and
$d_{j+1}=D_{j+1}d_j$.  Induction therefore gives
\begin{equation}\label{eq:ray-lower}
 Q_j(1,0,t)\geq \gamma_jt^{o_j},
 \qquad
 \gamma_1=\varepsilon_1,
 \qquad
 \gamma_{j+1}=\varepsilon_{j+1}\gamma_j^{k_{j+1}}
\end{equation}
for constants $\gamma_j>0$.  Put
\[
 \sigma_j=-\frac{1}{d_j}\log \gamma_j\geq0.
\]
Indeed, $0<\gamma_j\leq1$ follows inductively from
$0<\varepsilon_j\leq1/4$ and $k_j\geq1$.

At this point $p_j$ and $\sigma_j$ are known.  Since $p_j$ is continuous on
$\Pj^2\setminus U_j$, choose $L_j$ so that
\begin{equation}\label{eq:Lj-choice}
 L_j\geq L_{j-1}+2,
 \qquad
 L_j\geq j^2(\sigma_j+1),
 \qquad
 L_j>-\min_{\Pj^2\setminus U_j}p_j.
\end{equation}
Then put
\begin{equation}\label{eq:epsilon-recursion}
 R_j=e^{-d_jL_j},
 \qquad
 \varepsilon_{j+1}=\frac14R_j^{D_{j+1}-k_{j+1}}.
\end{equation}
This completes stage $j$: $L_j$ depends only on the first $j$ maps and
then determines $R_j$ and the perturbation in the $(j+1)$st map.  In
particular, no datum from $\Phi_{j+1}$ is used to define
$\varepsilon_{j+1}$.

Define
\begin{equation}\label{eq:Vj}
 V_j=\rho+\max\{p_j,-L_j\}
     =\max\{q_j,\rho-L_j\}.
\end{equation}
Both branches in the last maximum are $\omega_{\FS}$-psh.  Moreover,
\[
 \exp(q_j)=\frac{\mathfrak r_j^{1/d_j}}{\norm{X}_2}
\]
is continuous and vanishes only at $a$, whereas
$\exp(\rho-L_j)$ is positive there.  Hence
$q_j<\rho-L_j$ near $a$, so $V_j=\rho-L_j$ on a neighborhood of $a$.
It follows that $V_j$ is a continuous $\omega_{\FS}$-psh function on
$\Pj^2$; in particular, it is bounded.

\begin{lem}\label{lem:comparison}
For every $j\geq1$,
\begin{equation}\label{eq:visible-comparison}
 -\frac{\log4}{d_{j+1}}
 \leq p_{j+1}-p_j\leq0
 \quad\text{on }\{p_j\geq-L_j\},
\end{equation}
and
\begin{equation}\label{eq:global-comparison}
 p_{j+1}\leq\max\{p_j,-L_j\}
 \quad\text{on }\Pj^2.
\end{equation}
Consequently, $V_{j+1}\leq V_j$.
\end{lem}

\begin{proof}
Write $D=D_{j+1}$ and $k=k_{j+1}$.  Normalize
$\norm{X}_\infty=1$ and set $s=\mathfrak r_j(X)$.  The recursion is
\[
 P_{j+1}=\frac12Q_j^D,
 \qquad
 Q_{j+1}=\frac12P_j^D+
 \varepsilon_{j+1}Q_j^kT^{d_j(D-k)}.
\]
If $s\geq R_j$,
the definition \eqref{eq:epsilon-recursion} gives
\[
 \varepsilon_{j+1}s^k|T|^{d_j(D-k)}
 \leq\frac14s^D.
\]
If the maximum $s$ is attained by $Q_j$, its pure-power output has
modulus $s^D/2$.  If it is attained by $P_j$, the reverse triangle
inequality applied to the perturbed output gives a lower bound
$s^D/4$.  Every output is at most $3s^D/4$.  Hence
\[
 \frac14s^D\leq \mathfrak r_{j+1}(X)\leq\frac34s^D.
\]
After taking logarithms and using $d_{j+1}=D_{j+1}d_j$, this proves
\eqref{eq:visible-comparison}.

If $s<R_j$, write $s=R_j\tau$ with $0\leq\tau<1$.  The pure outputs and the
perturbed output are bounded by
\[
 R_j^D\left(\frac12\tau^D+\frac14\tau^k\right)<R_j^D.
\]
Thus $p_{j+1}<-L_j$ on this region.  Together with the upper half of
\eqref{eq:visible-comparison}, this proves
\eqref{eq:global-comparison}.  Since $-L_{j+1}\leq-L_j$, taking maxima
in \eqref{eq:global-comparison} proves $V_{j+1}\leq V_j$.
\end{proof}

Define the cutoff core
\[
 K_j=\{p_j\leq-L_j\}.
\]
Since
\[
 \exp(p_j)=\frac{\mathfrak r_j^{1/d_j}}{\norm{X}_\infty}
\]
extends continuously by zero at $a$, the set $K_j$ is closed and contains
a neighborhood of $a$.
The last condition in \eqref{eq:Lj-choice} gives $K_j\subset U_j$.
Consequently, every compact subset of $\Pj^2\setminus\{a\}$ is
disjoint from $K_j$ for all sufficiently large $j$.

\begin{prop}\label{prop:limit}
The decreasing sequence $V_j$ has a limit
\begin{equation}\label{eq:G-limit}
 G=\lim_{j\to\infty}V_j
 \in\PSH(\Pj^2,\omega_{\FS}).
\end{equation}
The function $G$ is continuous and finite on $\Pj^2\setminus\{a\}$,
and
\[
 G(z)\longrightarrow-\infty\qquad(z\to a).
\]
\end{prop}

\begin{proof}
Let $K\Subset\Pj^2\setminus\{a\}$.  The inclusion $K_j\subset U_j$ gives
$p_j>-L_j$ on $K$ for every sufficiently large $j$.  Hence
$V_j=q_j$ there, and \eqref{eq:visible-comparison} gives
\[
 \norm{q_{j+1}-q_j}_{L^\infty(K)}
 \leq\frac{\log4}{d_{j+1}}.
\]
Since $D_j\geq2$, the series $\sum_jd_j^{-1}$ converges.  Thus $q_j$,
and hence $V_j$, converges uniformly on $K$ to a continuous finite
function.  The decreasing global limit is not identically $-\infty$ and
is therefore $\omega_{\FS}$-psh.  At $a$ we have
$p_j(a)=-\infty$ and $\rho(a)=0$, so $V_j(a)=-L_j\to-\infty$.

To prove that $G(z)\to-\infty$ as $z\to a$, fix $M>0$ and choose $j$
with $L_j>M+1$.
Continuity and $V_j(a)=-L_j$ give $V_j<-M$ near $a$.  Since $G\leq V_j$,
the same holds for $G$.  Letting $M\to\infty$ proves the claim.  In
particular, $\exp(G)$ extends continuously to $a$ by zero.
\end{proof}

For later use, set
\[
 \Delta_j=\sum_{\ell>j}\frac{\log4}{d_\ell}.
\]
Since $D_\ell\geq2$, one has $\Delta_j\downarrow0$.

\section{The limiting Green function}\label{sec:limiting-green}

\subsection{Concentration of the Monge--Amp\`ere measure}
\label{sec:measure}

The preceding construction gives a limit with one isolated pole.  We now
identify its classical Monge--Amp\`ere measure.

\begin{prop}\label{prop:MA-limit}
The function $G$ in \eqref{eq:G-limit} belongs to
$DMA(\Pj^2,\omega_{\FS})$ and satisfies
\begin{equation}\label{eq:MA-G}
 (\omega_{\FS}+\ddc G)^2=\delta_a.
\end{equation}
\end{prop}

\begin{proof}
On the open set $\Pj^2\setminus K_j$, the cutoff in \eqref{eq:Vj} is
inactive and $V_j=q_j$.  By locality for bounded psh functions and
\eqref{eq:finite-stage-local},
\[
 (\omega_{\FS}+\ddc V_j)^2=0
 \quad\text{on }\Pj^2\setminus K_j.
\]
The function $V_j$ is bounded, so the Bedford--Taylor product has total
mass
\[
 \int_{\Pj^2}(\omega_{\FS}+\ddc V_j)^2
 =\int_{\Pj^2}\omega_{\FS}^2=1.
\]
It is therefore a probability measure $\mu_j$ supported on $K_j$.  For
every $\chi\in C^0(\Pj^2)$,
\[
 \left|\int_{\Pj^2}\chi\,d\mu_j-\chi(a)\right|
 \leq\sup_{x\in K_j}|\chi(x)-\chi(a)|
 \leq\sup_{x\in U_j}|\chi(x)-\chi(a)|\longrightarrow0,
\]
because $K_j\subset U_j=B_{\omega_{\FS}}(a,\eta_j)$ and
$\eta_j\to0$.  Hence
\begin{equation}\label{eq:MA-Vj}
 (\omega_{\FS}+\ddc V_j)^2\rightharpoonup\delta_a.
\end{equation}

By \Cref{prop:limit}, $G$ is locally bounded outside $a$, so
\Cref{lem:isolated-DMA} gives
$G\in DMA(\Pj^2,\omega_{\FS})$.  Since $V_j\downarrow G$ is a bounded
decreasing approximation, \Cref{def:DMA} also gives
\[
 (\omega_{\FS}+\ddc V_j)^2
 \rightharpoonup(\omega_{\FS}+\ddc G)^2.
\]
Comparison with \eqref{eq:MA-Vj} proves \eqref{eq:MA-G}.
\end{proof}

\subsection{The exact Lelong number}\label{sec:lelong}

It remains to show that the moving truncations preserve the prescribed
limiting vanishing order.  The persistence claim and real-ray witness
follow the mechanism of \cite[Section~3.4]{LX26}; the variable orders and
the Schwarz estimate with multiplicity below are what turn that mechanism
into exact control of an arbitrary prescribed $\lambda$.

\begin{prop}\label{prop:exact-lelong}
The limit $G$ satisfies
\begin{equation}\label{eq:exact-lelong}
 \nu(G,a)=\lambda.
\end{equation}
\end{prop}

\begin{proof}
\emph{Lower bound.}
In the chart $T=1$, set
\[
 \chi_{\FS}(\zeta)=\frac12\log(1+\norm{\zeta}_2^2).
\]
Define the local psh potential
\[
 g(\zeta):=G([1:\zeta])+\chi_{\FS}(\zeta).
\]
Then $g$ is psh near the origin and
$\rho([1:\zeta])=-\chi_{\FS}(\zeta)$ for $\zeta\in\D^2$.  Let
\[
 \mathcal H_j=(P_j(1,\cdot),Q_j(1,\cdot))
\]
be the affine germ of the non-leading coordinates.  Each one-step map
\eqref{eq:affine-map} sends the closed unit polydisc into the open unit
polydisc, so the same is true of $\mathcal H_j$.  Every component of
$\mathcal H_j$ vanishes to order at least $o_j$.  For $\zeta\neq0$, set
$e=\zeta/\norm{\zeta}_\infty$.  Each component of
$\tau\mapsto\mathcal H_j(\tau e)$ maps the unit disc into itself and
vanishes to order at least $o_j$.  The Schwarz lemma with multiplicity
gives
\begin{equation}\label{eq:Schwarz}
 \norm{\mathcal H_j(\zeta)}_\infty
 \leq\norm{\zeta}_\infty^{o_j},
 \qquad \zeta\in\D^2.
\end{equation}
For every fixed $\zeta\in\D^2\setminus\{0\}$, the cutoff is inactive for
all large $j$, and hence $q_j([1:\zeta])\to G([1:\zeta])$.
Since $o_j/d_j=\lambda_j$, \eqref{eq:Schwarz} gives
\[
 q_j([1:\zeta])+\chi_{\FS}(\zeta)
 \leq\lambda_j\log\norm{\zeta}_\infty.
\]
Letting $j\to\infty$ gives
$g(\zeta)\leq\lambda\log\norm{\zeta}_\infty$.  Dividing by the negative
number
$\log\norm{\zeta}_\infty$, and then letting $\zeta\to0$ in
\eqref{eq:lelong-liminf} gives
$\nu(G,a)=\nu(g,0)\geq\lambda$.

\emph{Upper bound.}
We first establish the persistence estimate used below.  If
$x\in\Pj^2$, $\Delta_j<2$, and $p_j(x)>-L_j$, then
\begin{equation}\label{eq:persistence}
 q_j(x)-\Delta_j\leq G(x)\leq q_j(x).
\end{equation}
Indeed, set
\[
 S_\ell=\sum_{s=j+1}^{\ell}\frac{\log4}{d_s},
 \qquad S_j=0.
\]
We prove by induction on $\ell\geq j$ that
\[
 p_\ell(x)\geq p_j(x)-S_\ell
 \quad\text{and}\quad
 p_\ell(x)>-L_\ell.
\]
The case $\ell=j$ is the hypothesis.  If the assertion holds at stage
$\ell$, then the cutoff is inactive there, and the lower estimate in
\eqref{eq:visible-comparison}, with index $\ell$, gives
\[
 p_{\ell+1}(x)
 \geq p_\ell(x)-\frac{\log4}{d_{\ell+1}}
 \geq p_j(x)-S_{\ell+1}.
\]
Because $S_{\ell+1}<\Delta_j<2$ and
$L_{\ell+1}\geq L_j+2(\ell+1-j)$, this also gives
\[
 p_{\ell+1}(x)>-L_j-2\geq-L_{\ell+1}.
\]
The induction is complete: once the cutoff is inactive at stage $j$, it
remains inactive.  Summing both sides of
\eqref{eq:visible-comparison} and passing to the limit now proves
\eqref{eq:persistence}.

Put
\begin{equation}\label{eq:witness}
 s_j=\frac{L_j}{j},
 \qquad
 t_j=e^{-s_j},
 \qquad
 x_j=[1:0:t_j].
\end{equation}
By \eqref{eq:ray-lower},
\begin{equation}\label{eq:witness-visible}
 p_j(x_j)\geq-\lambda_js_j-\sigma_j.
\end{equation}
The middle inequality in \eqref{eq:Lj-choice} implies $s_j\geq j$ and
\[
 \lambda_js_j+\sigma_j
 \leq \frac{L_j}{j}+\frac{L_j}{j^2}<L_j
\]
for all sufficiently large $j$.  Hence $p_j(x_j)>-L_j$.  Since
$\Delta_j<2$ for large $j$, \eqref{eq:persistence} and
\eqref{eq:witness-visible} yield
\begin{equation}\label{eq:witness-final}
 g(0,t_j)
 \geq-\lambda_js_j-\sigma_j-\Delta_j,
\end{equation}
because $\rho(x_j)+\chi_{\FS}(0,t_j)=0$.  Moreover,
\begin{equation}\label{eq:error-ratio}
 s_j\longrightarrow\infty,
 \qquad
 \frac{\sigma_j}{s_j}\leq\frac1j,
 \qquad
 0\leq\frac{\Delta_j}{s_j}\leq\frac{\Delta_j}{j}\longrightarrow0.
\end{equation}
Divide \eqref{eq:witness-final} by
$\log t_j=-s_j<0$; the inequality reverses and gives
\[
 \frac{g(0,t_j)}{\log t_j}
 \leq\lambda_j+\frac{\sigma_j}{s_j}+\frac{\Delta_j}{s_j}.
\]
The points $(0,t_j)$ tend to the origin.  Applying
\eqref{eq:lelong-liminf} to $g$ and using \eqref{eq:error-ratio} and
$\lambda_j\to\lambda$, we obtain
\[
 \nu(G,a)=\nu(g,0)\leq
 \liminf_{j\to\infty}
 \frac{g(0,t_j)}{\log t_j}
 \leq\limsup_{j\to\infty}
 \frac{g(0,t_j)}{\log t_j}
 \leq\lambda.
\]
Together with the lower bound, this proves \eqref{eq:exact-lelong}.
\end{proof}

For the fixed parameter $0\leq\lambda<1$, set
\begin{equation}\label{eq:G2lambda}
 G_{2,\lambda}:=G,
\end{equation}
where $G$ is the limit in \eqref{eq:G-limit}.  By
\Cref{prop:limit,prop:MA-limit,prop:exact-lelong}, this function has the
single pole $a$, belongs to $DMA(\Pj^2,\omega_{\FS})$, satisfies
\[
 (\omega_{\FS}+\ddc G_{2,\lambda})^2=\delta_a,
 \qquad \nu(G_{2,\lambda},a)=\lambda,
\]
is continuous off $a$, and has a continuous exponential after extension
by zero at $a$.

\subsection{The endpoint at one}\label{sec:endpoints}

For $\lambda=1$, define directly on $\Pj^2$
\[
 G_{2,1}([T:Z:W])
 =\log\max\{|Z|,|W|\}-\log\norm{(T,Z,W)}_2.
\]
This is the degree-one projection weight centered at $a$: in
\cite[Section~2.2.1 and Theorem~2.4]{CG09}, take the projection
$[T:Z:W]\dashrightarrow[Z:W]$ and the bounded
$\omega_{\FS}$-psh target potential
\[
 [Z:W]\longmapsto
 \log\max\{|Z|,|W|\}-\log\norm{(Z,W)}_2.
\]
The cited theorem gives the global Monge--Amp\`ere equality
\[
 (\omega_{\FS}+\ddc G_{2,1})^2=\delta_a.
\]
Since $G_{2,1}$ is locally bounded off $a$,
\Cref{lem:isolated-DMA} places it in the global domain
$DMA(\Pj^2,\omega_{\FS})$, with the same measure.  The displayed formula
also gives $\nu(G_{2,1},a)=1$, continuity away from $a$, and continuity of
its exponential after extension by zero.

\section{Projective suspension}\label{sec:suspension}

The next construction raises the dimension without changing the Lelong
number.  Its max-product step is the homogeneous projective counterpart
of the local higher-dimensional construction in
\cite[Section~3.5]{LX26}.  We first isolate the local measure calculation,
including the hypotheses needed for B{\l}ocki's formula.

\begin{lem}[Local max-product]\label{lem:local-max-product}
Let $U\subset\C^p$ and $V\subset\C^q$ be domains containing the origin.
Let $u\in\PSH(U)$ and $v\in\PSH(V)$, regard them as functions in separate
variables on $U\times V$, and put
\[
 w(z,\eta)=\max\{u(z),v(\eta)\}.
\]
Assume that
\[
 u\in\mathcal D(U),\qquad
 v\in\mathcal D(V),\qquad
 w\in\mathcal D(U\times V),
\]
that $u(0)=v(0)=-\infty$, and that
\[
 (\ddc u)^p=m_u\delta_0,
 \qquad
 (\ddc v)^q=m_v\delta_0.
\]
Then
\begin{equation}\label{eq:local-max-product}
 (\ddc w)^{p+q}=m_um_v\delta_{(0,0)}.
\end{equation}
\end{lem}

\begin{proof}
For $T>0$, set
\[
 u_T=\max\{u,-T\}+T,
 \qquad
 v_T=\max\{v,-T\}+T.
\]
These are nonnegative locally bounded psh functions.  Put
\[
 E_T^u=\{u>-T\}=\{u_T>0\}.
\]
Since $u$ is plurifinely continuous and Borel measurable, $E_T^u$ is
plurifinely open and Borel.  Moreover,
$u_T\in\mathcal D(U)$ because $u_T$ is locally bounded, whereas
$u+T\in\mathcal D(U)$ because $u\in\mathcal D(U)$ and this class is
invariant under addition of constants.  The two functions agree on
$E_T^u$.  El Kadiri's plurifine locality theorem
\cite[Corollary~3.7]{EK23} therefore gives
\[
 \mathbf1_{E_T^u}(\ddc u_T)^p
 =\mathbf1_{E_T^u}(\ddc u)^p=0,
\]
because $(\ddc u)^p=m_u\delta_0$ and $0\notin E_T^u$.  The same
argument for $v$ gives
\[
 \mathbf1_{\{v_T>0\}}(\ddc v_T)^q=0.
\]
Consequently,
\[
 \int_{\{u_T>0\}}(\ddc u_T)^p
 =\int_{\{v_T>0\}}(\ddc v_T)^q=0.
\]

These are precisely the two support hypotheses in B{\l}ocki's
max-product theorem \cite[Theorem~7]{Blo00}.  Applied to the
nonnegative locally bounded functions $u_T$ and $v_T$ in separate
variables, that theorem yields, with $\operatorname{pr}_U$ and
$\operatorname{pr}_V$ denoting the two projections,
\begin{equation}\label{eq:Blocki-local-product}
 \bigl(\ddc\max\{u_T,v_T\}\bigr)^{p+q}
 =\operatorname{pr}_U^*(\ddc u_T)^p
  \wedge\operatorname{pr}_V^*(\ddc v_T)^q
 =(\ddc u_T)^p\otimes(\ddc v_T)^q.
\end{equation}
Here the last equality identifies, in separated variables, the
pulled-back top-degree wedge product with the product measure; in
particular, $\delta_0\otimes\delta_0=\delta_{(0,0)}$.  The coefficient
is one: \eqref{eq:Blocki-local-product} is B{\l}ocki's max-product
identity, not the binomial expansion of the Monge--Amp\`ere measure of
a sum.

Let $\mu_T=(\ddc u_T)^p$ and $\nu_T=(\ddc v_T)^q$.  Constants do not
affect $\ddc$, and the canonical truncations
$\max\{u,-T\}\downarrow u$ and $\max\{v,-T\}\downarrow v$ for
$T\in\mathbf N$.  The defining continuity of $\mathcal D$ gives
\[
 \mu_T\rightharpoonup m_u\delta_0,
 \qquad
 \nu_T\rightharpoonup m_v\delta_0
\]
locally as $T\to\infty$ through the positive integers.

To pass to exterior products, fix $\chi\in C_c^0(U\times V)$ and choose
compactly supported cutoffs in $U$ and $V$ equal to one on the coordinate
projections of $\operatorname{supp}\chi$.  On a product compact,
Stone--Weierstrass then reduces the argument to separated tests
$f(z)g(\eta)$ with $f\in C_c^0(U)$ and $g\in C_c^0(V)$.  Local weak
convergence uniformly bounds the masses there, while
\[
 \int f(z)g(\eta)\,d(\mu_T\otimes\nu_T)
 =\left(\int f\,d\mu_T\right)
  \left(\int g\,d\nu_T\right).
\]
Uniform approximation therefore gives, locally on $U\times V$,
\[
 \mu_T\otimes\nu_T
 \rightharpoonup m_um_v\,\delta_0\otimes\delta_0
\]
against every test function in $C_c^0(U\times V)$.

Finally,
\[
 \max\{u_T,v_T\}=\max\{w,-T\}+T.
\]
Since $w\in\mathcal D(U\times V)$, continuity along its canonical
truncations gives, locally on $U\times V$,
\[
 \bigl(\ddc\max\{w,-T\}\bigr)^{p+q}
 \rightharpoonup(\ddc w)^{p+q}.
\]
Passing to the limit in \eqref{eq:Blocki-local-product} proves
\eqref{eq:local-max-product}.
\end{proof}

\begin{lem}[Projective suspension]\label{lem:projective-suspension}
Let $a=[1:0:0]\in\Pj^2$, and let
$G\in\PSH(\Pj^2,\omega_{\FS})\cap DMA(\Pj^2,\omega_{\FS})$ satisfy
\[
 (\omega_{\FS}+\ddc G)^2=\delta_a,
 \qquad G^{-1}(-\infty)=\{a\}.
\]
Assume that $G$ is continuous on $\Pj^2\setminus\{a\}$ and that
$G(z)\to-\infty$ as $z\to a$.  For every $m\geq1$ there is
\[
 \mathcal S_m(G)\in
 \PSH(\Pj^{m+2},\omega_{\FS})\cap DMA(\Pj^{m+2},\omega_{\FS})
\]
with a single pole at $a_m=[1:0:\cdots:0]$ such that
\begin{equation}\label{eq:suspension-conclusions}
 (\omega_{\FS}+\ddc\mathcal S_m(G))^{m+2}=\delta_{a_m},
 \qquad
 \nu(\mathcal S_m(G),a_m)=\nu(G,a).
\end{equation}
The function is continuous off $a_m$, and its exponential extends
continuously by zero at $a_m$.
\end{lem}

\begin{proof}
\emph{Plurisubharmonicity and the pole set.}
For $\xi\in\C^3\setminus\{0\}$, set
\[
 \widehat G(\xi)=G([\xi])+\log\norm{\xi}_2,
 \qquad \widehat G(0)=-\infty.
\]
On $\C^3\setminus\{0\}$,
\[
 \ddc\widehat G=\pi_2^*(\omega_{\FS}+\ddc G)\geq0.
\]
Thus $\widehat G$ is psh away from the origin and is bounded above there by
$\log\norm{\xi}_2+\sup_{\Pj^2}G$.  The removable-singularity theorem
therefore makes it psh on $\C^3$.

For $\eta\in\C^m$, with $\log\norm{0}_\infty=-\infty$, put
\[
 W(\xi,\eta)=
 \max\{\widehat G(\xi),\log\norm{\eta}_\infty\}.
\]
Since $W(\tau\xi,\tau\eta)=W(\xi,\eta)+\log|\tau|$, it defines
\begin{equation}\label{eq:projective-suspension}
 \mathcal S_m(G)([\xi:\eta])
 =W(\xi,\eta)-\log\norm{(\xi,\eta)}_2.
\end{equation}
The identity
\[
 \pi_{m+2}^*(\omega_{\FS}+\ddc\mathcal S_m(G))=\ddc W\geq0
\]
proves that $\mathcal S_m(G)$ is $\omega_{\FS}$-psh.  The function
$\widehat G$ is $-\infty$ precisely on the cone over $a$, whereas
$\log\norm{\eta}_\infty=-\infty$ precisely when $\eta=0$.  Hence
$W^{-1}(-\infty)$ projects to the single point $a_m$.

\emph{Continuity.}
Set $\mathcal E_G(\xi)=\norm{\xi}_2\exp G([\xi])$ for $\xi\neq0$ and
$\mathcal E_G(0)=0$.
The estimate
$\mathcal E_G(\xi)\leq\norm{\xi}_2\exp(\sup G)$ controls
$\mathcal E_G$ at the origin.  Now
\[
 \exp\mathcal S_m(G)([\xi:\eta])
 =\frac{\max\{\mathcal E_G(\xi),\norm{\eta}_\infty\}}
 {\norm{(\xi,\eta)}_2}.
\]
Away from the cone over $a$, continuity of $\mathcal E_G$ follows from
that of $G$.
As a nonzero point of that cone is approached, the assumption
$G(z)\to-\infty$ gives $\mathcal E_G(\xi)\to0$, while continuity at the origin follows
from the preceding estimate.  Hence the displayed formula shows that
$\exp\mathcal S_m(G)$ is continuous everywhere after taking the value zero
at $a_m$.  It is positive off $a_m$, so its logarithm
$\mathcal S_m(G)$ is continuous there and hence locally bounded on the
complement of $a_m$.

\emph{The Monge--Amp\`ere measure.}
In the affine chart centered at $a_m$, write
\[
 u(z)=G([1:z])+\frac12\log(1+\norm{z}_2^2),
 \qquad v(\eta)=\log\norm{\eta}_\infty.
\]
The local psh potential in \eqref{eq:projective-suspension} is
$w(z,\eta)=\max\{u(z),v(\eta)\}$.  Each of $u$, $v$, and $w$ is
locally bounded away from the origin and is the affine potential of a
projective psh function.  For $v$, the corresponding projective function is
\[
 [s:\eta]\longmapsto
 \log\norm{\eta}_\infty-\log\norm{(s,\eta)}_2
 \quad\text{on }\Pj^m.
\]
The projective functions corresponding to $u$, $v$, and $w$ are,
respectively, $G$, the function in the preceding display, and
$\mathcal S_m(G)$.  Each is locally bounded off its unique pole and
therefore bounded near a hyperplane avoiding that pole.  Hence
\cite[Proposition~4.6]{CGZ08} places all three in
$DMA_{\mathrm{loc}}$.  Choose sufficiently small polydiscs
\[
 U_0\Subset\C^2,
 \qquad V_0\Subset\C^m
\]
centered at the origin.  After shrinking them if necessary,
\cite[Definition~3.1]{CGZ08} gives
\[
 u\in\mathcal D(U_0),\qquad
 v\in\mathcal D(V_0),\qquad
 w\in\mathcal D(U_0\times V_0).
\]
The inclusion \eqref{eq:DMA-chain} also places the three projective
functions in global $DMA$.

For later use, we record local--global compatibility.  For any one of
these projective functions $\psi$, take the canonical bounded global
approximants
\[
 \psi_\ell=\max\{\psi,-\ell\}\downarrow\psi.
\]
Both $\psi$ and the constant $-\ell$ are $\omega_{\FS}$-psh, so their
maximum is again $\omega_{\FS}$-psh.
If $\rho_{\FS}$ is a local potential on an affine chart $A$, then
\[
 \ddc(\psi_\ell+\rho_{\FS})=\omega_{\FS}+\ddc\psi_\ell.
\]
Thus local $\mathcal D$ convergence and global $DMA$ convergence, tested
against $\chi\in C_c^0(A)$, identify the local Monge--Amp\`ere measure
with the restriction of the global one.

The assumed global equation for $G$ now gives
\[
 (\ddc u)^2=\delta_0.
\]
For $m=1$, the normalization \eqref{eq:ddc} gives directly
$\ddc v=\delta_0$.  For $m\geq2$, the projective function
\[
 [s:\eta]\longmapsto
 \log\norm{\eta}_\infty-\log\norm{(s,\eta)}_2
\]
is the degree-one projection weight of
\cite[Section~2.2.1 and Theorem~2.4]{CG09}, with the bounded
$\omega_{\FS}$-psh target potential
\[
 [\eta]\longmapsto
 \log\norm{\eta}_\infty-\log\norm{\eta}_2
 \quad\text{on }\Pj^{m-1}.
\]
This target potential is bounded and nonpositive, since
$\norm{\eta}_\infty\leq\norm{\eta}_2$.  Hence the hypotheses of
\cite[Theorem~2.4]{CG09} are satisfied.  The degree-one projection has
the single indeterminacy point $[1:0:\cdots:0]$; since the masses in that
theorem sum to one, its global Monge--Amp\`ere measure is
$\delta_{[1:0:\cdots:0]}$.  Local--global compatibility gives in every
case
\[
 (\ddc v)^m=\delta_0.
\]
Applying \Cref{lem:local-max-product} on $U_0\times V_0$, with $p=2$
and $q=m$, gives
\[
 (\ddc w)^{m+2}=\delta_{(0,0)}.
\]

Local--global compatibility shows that the global measure of
$\mathcal S_m(G)$ has an atom of mass one at $a_m$.  This positive
measure has total mass $\int_{\Pj^{m+2}}\omega_{\FS}^{m+2}=1$; hence it
is $\delta_{a_m}$.  This proves the first identity in
\eqref{eq:suspension-conclusions}.

\emph{The Lelong number.}
Let
\[
 M_f(r)=\sup_{\norm{\zeta}_\infty\leq r}f(\zeta),
\]
using the product maximum norm for $w$.  By the Hadamard three-circles
theorem, $t\mapsto M_f(e^t)$ is convex for $t\ll0$.  Its slope at
$-\infty$ therefore exists; the maximum principle together with
\eqref{eq:lelong-liminf} identifies this slope with the Lelong number:
\[
 \nu(f,0)=\lim_{r\downarrow0}\frac{M_f(r)}{\log r}.
\]
Since $M_v(r)=\log r$ and
$M_w(r)=\max\{M_u(r),\log r\}$, division by $\log r<0$ gives
\[
 \nu(w,0)=\min\{\nu(u,0),1\}.
\]
Because $w$ and $u$ are the affine psh potentials of
$\mathcal S_m(G)$ and $G$, respectively, smooth Fubini--Study
potentials do not change their Lelong numbers.  Hence
\[
 \nu(\mathcal S_m(G),a_m)
 =\min\{\nu(G,a),1\}
 =\nu(G,a),
\]
where the last equality follows from
\cite[Proposition~2.1]{CG09}.
\end{proof}

\begin{proof}[Proof of \Cref{thm:main} and \Cref{cor:range}]
For $n=2$, the functions in \eqref{eq:G2lambda} and the endpoint
$G_{2,1}$ have all the properties asserted in \Cref{thm:main}.  If $n>2$
and $0\leq\lambda\leq1$, define
\[
 G_{n,\lambda}=\mathcal S_{n-2}(G_{2,\lambda}).
\]
The continuous extension of $\exp(G_{2,\lambda})$ takes the value zero
at the pole, so $G_{2,\lambda}$ tends to $-\infty$ there, as required in
the suspension lemma.
The projective suspension lemma gives the required equation, pole set,
Lelong number, and continuity properties.  This proves
\Cref{thm:main}.

For \Cref{cor:range}, Lelong numbers are nonnegative, while every positive
closed current in the normalized Fubini--Study class has Lelong number at
most one \cite[Proposition~2.1]{CG09}.  Conversely, let $b\in\Pj^n$ be
prescribed.  Choose a projective unitary transformation $U$ with
$U(b)=[1:0:\cdots:0]$.  Since $U^*\omega_{\FS}=\omega_{\FS}$, the function
$G_{n,\lambda}\circ U$ has pole $b$, Monge--Amp\`ere measure $\delta_b$,
and Lelong number $\lambda$ there.  Thus every value in $[0,1]$ occurs at
$b$, and no other value can occur.
\end{proof}

\section{Seshadri obstructions and finite pullbacks}
\label{sec:finite-pullbacks}

We now turn from projective space to compact K\"ahler manifolds and then
to projective applications.  The Seshadri constant gives an upper bound
for the one-pole Lelong range, but its endpoint need not be attained.

\subsection{The K\"ahler Seshadri bound}\label{subsec:seshadri-bound}

We begin with the general upper bound, using the Seshadri constant and
one-pole range defined in the Introduction.

\begin{prop}[K\"ahler Seshadri bound]\label{prop:seshadri-obstruction}
Let $X$ be a compact K\"ahler manifold, let $\alpha$ be a
volume-normalized K\"ahler class, and let $x\in X$.  Then
\begin{equation}\label{eq:range-upper-bound}
 \mathcal R_\alpha(x)\subset[0,\varepsilon(\alpha,x)].
\end{equation}
\end{prop}

\begin{proof}
Choose a K\"ahler representative $\vartheta\in\alpha$ as in
\eqref{eq:projective-range}.
Formula~(3) of \cite{CG09} expresses $\varepsilon(\alpha,x)$ as the
supremum of the Lelong numbers of $\alpha$-psh potentials that are
locally bounded on a punctured neighborhood of $x$.  Every function in
$\mathcal R_\alpha(x)$ has this property, which proves
\eqref{eq:range-upper-bound}.
\end{proof}

\subsection{A nonattained endpoint}\label{subsec:nonattained-endpoint}

\begin{proof}[Proof of \Cref{thm:koike-obstruction}]
The proof has three steps.  We first choose a degree-one del Pezzo
surface and a point on a nodal anticanonical curve.  We then compute the
Seshadri constant at that point.  Finally, Koike's theorem identifies the
only positive current in the boundary class on the blow-up.

\emph{The surface and the point.}
Let $C_0$ be an irreducible nodal cubic in $\Pj^2$.  We first verify that
eight smooth points in del Pezzo general position can be chosen on this
fixed curve.

\medskip\noindent
\textbf{General-position lemma.}
\emph{Every irreducible nodal plane cubic contains eight distinct smooth
points in del Pezzo general position.}

To prove the lemma, fix $o\in C_0^{\mathrm{reg}}$, and let
$\tau:\Pj^1\to C_0$ be the normalization, with $p_+$ and $p_-$ the two
points over the node.  The normalization description of the Picard group
\cite[Exercise~II.6.9]{Har77} gives the gluing-scalar isomorphism
\[
 \Gamma:\operatorname{Pic}^0(C_0)\longrightarrow\C^*.
\]
The Abel map
\[
 \mathcal A_o:C_0^{\mathrm{reg}}\longrightarrow
 \operatorname{Pic}^0(C_0),
 \qquad
 \mathcal A_o(y)=\mathcal O_{C_0}(y-o),
\]
is also an isomorphism.  Indeed, choose a coordinate on $\Pj^1$ sending
$p_+$, $p_-$, and $\tau^{-1}(o)$ to $0$, $\infty$, and $1$.  The gluing
scalar of $\mathcal O_{C_0}(y-o)$ is then the coordinate of
$\tau^{-1}(y)$, up to inversion.  Write
\[
 \zeta(y)=\Gamma(\mathcal A_o(y))\in\C^*
\]
and, for $\ell=1,2,3$, set
\[
 \kappa_\ell=
 \Gamma\!\left(
  \mathcal O_{C_0}(\ell)\otimes
  \mathcal O_{C_0}(-3\ell o)
 \right)\in\C^*.
\]
If $D=\sum_\mu m_\mu y_\mu$ is supported on
$C_0^{\mathrm{reg}}$ and has degree $3\ell$, then
\begin{equation}\label{eq:nodal-cubic-group-law}
 \mathcal O_{C_0}(D)\simeq\mathcal O_{C_0}(\ell)
 \quad\Longleftrightarrow\quad
 \prod_\mu\zeta(y_\mu)^{m_\mu}=\kappa_\ell.
\end{equation}

We apply \eqref{eq:nodal-cubic-group-law} to the three incidence
conditions in the del Pezzo general-position criterion.  Three distinct
points $y_i,y_j,y_k$ are collinear exactly when
\[
 \mathcal O_{C_0}(y_i+y_j+y_k)\simeq\mathcal O_{C_0}(1),
 \quad\text{equivalently}\quad
 \zeta(y_i)\zeta(y_j)\zeta(y_k)=\kappa_1.
\]
Likewise, six points indexed by $I\subset\{1,\ldots,8\}$ lie on a conic
exactly when
\[
 \mathcal O_{C_0}\!\left(\sum_{i\in I}y_i\right)
 \simeq\mathcal O_{C_0}(2),
 \quad\text{equivalently}\quad
 \prod_{i\in I}\zeta(y_i)=\kappa_2.
\]
The geometric equivalences follow from the exact sequences
\[
 0\longrightarrow\mathcal O_{\Pj^2}(\ell-3)
 \longrightarrow\mathcal O_{\Pj^2}(\ell)
 \longrightarrow\mathcal O_{C_0}(\ell)
 \longrightarrow0,
 \qquad \ell=1,2,
\]
because the resulting maps on global sections are isomorphisms.

For each $i$, a cubic through all eight points and singular at $y_i$
exists exactly when
\begin{equation}\label{eq:singular-cubic-relation}
 \mathcal O_{C_0}\!\left(2y_i+\sum_{j\ne i}y_j\right)
 \simeq\mathcal O_{C_0}(3),
 \quad\text{equivalently}\quad
 \zeta(y_i)^2\prod_{j\ne i}\zeta(y_j)=\kappa_3.
\end{equation}
One implication follows from B\'ezout's theorem.  Such a cubic is not
$C_0$, because $C_0$ is smooth at $y_i$, and its intersection
multiplicity with $C_0$ is at least two at $y_i$ and at least one at
each of the other seven points.  The total intersection number is nine,
so the intersection divisor is the one in
\eqref{eq:singular-cubic-relation}.

Conversely, suppose the linear equivalence in
\eqref{eq:singular-cubic-relation} holds.  If $F_0$ defines $C_0$, the
exact sequence
\[
 0\longrightarrow\mathcal O_{\Pj^2}
 \xrightarrow{\cdot F_0}\mathcal O_{\Pj^2}(3)
 \longrightarrow\mathcal O_{C_0}(3)\longrightarrow0
\]
shows that a section with the prescribed divisor lifts to a cubic $F$.
Its restriction to $C_0$ vanishes to order at least two at $y_i$, so
$dF(y_i)$ vanishes on $T_{y_i}C_0$.  The $1$-dimensional annihilator of
$T_{y_i}C_0$ is spanned by $dF_0(y_i)$, so
$dF(y_i)=c\,dF_0(y_i)$ for some $c\in\C$.  The cubic
$F-cF_0$ has the same nonzero restriction to $C_0$, passes through all
eight points, and is singular at $y_i$.

After identifying $(C_0^{\mathrm{reg}})^8$ with $(\C^*)^8$, all bad
configurations are therefore contained in the finite union of loci
\[
 \zeta_i=\zeta_j,\qquad
 \prod_{i\in I}\zeta_i=\kappa_1\ (|I|=3),\qquad
 \prod_{i\in I}\zeta_i=\kappa_2\ (|I|=6),\qquad
 \zeta_i^2\prod_{j\ne i}\zeta_j=\kappa_3.
\]
Each is the zero locus of a nonconstant Laurent polynomial, hence a
proper closed subset of the irreducible variety $(\C^*)^8$.  Their finite
union cannot cover the parameter space.  We may therefore choose eight
distinct smooth points $y_1,\ldots,y_8$ outside it.  By the
general-position criterion
\cite[Proposition~8.1.25; see also the definition immediately preceding
it]{Dol12}, these points are in del Pezzo general position.  This proves
the general-position lemma.

Let $\pi:X\to\Pj^2$ be their blow-up.  Then $X$ is a degree-one del
Pezzo surface, and
\[
 \alpha=c_1(-K_X)=3H-\sum_{i=1}^8\{E_i\},
 \qquad \alpha^2=1,
\]
where $H=\pi^*c_1(\mathcal O_{\Pj^2}(1))$ and the $E_i$ are the
exceptional curves.  The strict transform $C_X$ of $C_0$ lies in
$|-K_X|$.

Blowing up smooth points of $C_0$ does not change its strict transform as
an abstract curve, so $\pi|_{C_X}:C_X\to C_0$ is an isomorphism.  We use
the preceding Abel identification also for $C_X$, and denote by $o$ the
corresponding base point.  Since
\[
 \deg N_{C_X/X}=C_X^2=1,
\]
the bundle
\[
 \mathscr N_x:=N_{C_X/X}\otimes\mathcal O_{C_X}(-x)
\]
belongs to $\operatorname{Pic}^0(C_X)$.  More precisely,
\[
 \mathscr N_x=
 \bigl(N_{C_X/X}\otimes\mathcal O_{C_X}(-o)\bigr)
 \otimes\mathcal A_o(x)^{-1}.
\]
Thus $x\mapsto\mathscr N_x$ is the Abel isomorphism followed by inversion
and translation, and is itself an isomorphism
\[
 C_X^{\mathrm{reg}}\longrightarrow\operatorname{Pic}^0(C_X).
\]
This is the degree-one normal-bundle identification used in
\cite[Section~4.2, especially p.~259]{Koi23}.

In the normalization-gluing coordinate
$\Gamma:\operatorname{Pic}^0(C_X)\simeq\C^*$, a degree-zero line bundle
is unitary flat exactly when its gluing scalar has modulus one
\cite[p.~233]{Koi23}.  Thus the unitary-flat locus is the proper real
subgroup $\mathrm U(1)\subset\C^*$.
Its inverse image under $x\mapsto\mathscr N_x$ is a proper real circle in
$C_X^{\mathrm{reg}}\simeq\C^*$.  Together with the finite set
$C_X\cap\bigcup_{i=1}^8E_i$, it cannot cover $C_X^{\mathrm{reg}}$.
We may therefore choose
\[
 x\in C_X^{\mathrm{reg}}\setminus\bigcup_{i=1}^8E_i
\]
so that $\mathscr N_x$ is not unitary flat.  Set $y_9=\pi(x)$.

Let $\sigma:Y\to X$ be the blow-up at $x$, let $E$ be its exceptional
curve, and let $C\subset Y$ be the strict transform of $C_X$.  We keep
$E_i$ for the strict transforms on $Y$ of the first eight exceptional
curves.  Identifying $C$ with $C_X$, the normal-bundle transformation
formula gives
\[
 N_{C/Y}=\mathcal O_Y(C)|_C
 \simeq N_{C_X/X}\otimes\mathcal O_{C_X}(-x)
 =\mathscr N_x,
\]
so $N_{C/Y}$ is not unitary flat.  Moreover,
\begin{equation}\label{eq:anticanonical-nine-blowup}
 \beta:=\sigma^*\alpha-\{E\}=c_1(-K_Y)=\{C\},
 \qquad \beta^2=0.
\end{equation}
\emph{The Seshadri constant.}
The class $\beta$ is nef.  Indeed, $\beta\cdot E=\beta\cdot E_i=1$ for
$1\leq i\leq8$, and $\beta\cdot C=0$.  Any other irreducible curve
$D\subset Y$ is the strict transform of a plane curve $D_0\neq C_0$ of
degree $d$ and
multiplicities $m_i$ at $y_i$, $1\leq i\leq9$.  B\'ezout's theorem gives
\[
 \beta\cdot D=3d-\sum_{i=1}^9m_i\geq0,
\]
because all nine points are smooth on $C_0$ and
$i_{y_i}(C_0,D_0)\geq m_i$.  Hence
$\varepsilon(\alpha,x)\geq1$.  Conversely, whenever
$\sigma^*\alpha-t\{E\}$ is nef,
\[
 0\leq(\sigma^*\alpha-t\{E\})\cdot C
 =\alpha\cdot C_X-t\,\operatorname{mult}_xC_X
 =1-t,
\]
because $C_X\in|-K_X|$ and $x$ is a smooth point of $C_X$.  Thus
$t\leq1$, and consequently $\varepsilon(\alpha,x)=1$.

\emph{Uniqueness in the boundary class.}
The composite $Y\to\Pj^2$ is the blow-up at the nine distinct smooth
points $y_1,\ldots,y_9$.  The divisor $C\in|-K_Y|$ is reduced and is an
irreducible nodal rational curve, hence a cycle of rational curves in
Koike's terminology.  We have proved that $-K_Y$ is nef, and
\[
 (-K_Y)|_C=\mathcal O_Y(C)|_C=N_{C/Y}
\]
is not unitary flat.  Thus every hypothesis of
\cite[Theorem~1.2]{Koi23} is satisfied.  Conditions~\textup{(ii)} and
\textup{(iv)} of that theorem are equivalent.  Koike's ``closed
semi-positive $(1,1)$-currents'' are precisely the positive closed
$(1,1)$-currents in our terminology, and condition~\textup{(iv)} says
that their set in $c_1(-K_Y)$ is not a singleton.  Since $N_{C/Y}$ is
not unitary flat, this set is a singleton.  It contains $[C]$, and hence
equals $\{[C]\}$.

Suppose that $T$ is a positive closed current with $\{T\}=\alpha$, whose
local potentials are locally bounded on $X\setminus\{x\}$, and that
$\nu(T,x)=1$.
We use the standard pullback of a positive closed $(1,1)$-current by the
modification $\sigma$: locally, if $T=\ddc u$ with $u$ psh, then
$\sigma^*T=\ddc(u\circ\sigma)$.  Under the present boundedness assumption,
$u\circ\sigma$ is not identically $-\infty$ on a coordinate
neighborhood, so this defines a positive closed current.
For the blow-up of a point, the proof of
\cite[Proposition~3.1, p.~256]{MX20}, stated for positive
$(1,1)$-currents, gives
\[
 \nu(\sigma^*T,E)=\nu(T,x)=1.
\]
The left-hand side is the generic Lelong number of
$\sigma^*T$ along $E$.
Siu's divisorial decomposition \cite{Siu74}, whose coefficient along a
prime divisor is its generic Lelong number, therefore gives
\[
 \sigma^*T=[E]+R,
 \qquad R\geq0,
 \qquad \{R\}=\sigma^*\alpha-\{E\}=\beta.
\]
The preceding paragraph forces $R=[C]$.  Hence $T=[C_X]$ on
$X\setminus\{x\}$, because $\sigma$ is biholomorphic away from $E$.
At any smooth point of $C_X\setminus\{x\}$, the current $T$ has Lelong
number one, contradicting the local boundedness of its potentials there.
This proves the theorem.
\end{proof}

Coman--Guedj constructed an endpoint Green function when the chosen point
is the ninth base point of the anticanonical pencil
\cite[Section~4.3.2]{CG09}.  The point above is chosen instead so that the
normal bundle is not unitary flat; Koike's equivalence then forces
nonattainment.

Thus the inclusion \eqref{eq:range-upper-bound} can be strict at its
right endpoint, already on a projective surface.  A single Seshadri
constant therefore does not determine the one-pole Lelong range.

The obstruction rules out a universal endpoint theorem.  We next isolate
a functorial mechanism that transfers the projective-space range to
varieties carrying a suitable finite map.

\subsection{Finite pullbacks}\label{subsec:finite-pullbacks}

Throughout this subsection, all complex manifolds are connected.
Let $F:X\to Y$ be a finite surjective holomorphic map between complex
$n$-folds, let $a\in Y$,
and let $x\in F^{-1}(a)$.  Write $e_x(F)$ for the local degree of $F$ at
$x$, defined by the local-algebra length
\[
 e_x(F)=\dim_\C
 \bigl(\mathcal O_{X,x}/\mathfrak m_a\mathcal O_{X,x}\bigr).
\]
This is the standard local multiplicity of a finite holomorphic map;
see \cite[Section~5.1]{AGV85}.
Thus $\sum_{z\in F^{-1}(a)}e_z(F)=\deg F$.
Let $\mathfrak m_a\subset\mathcal O_{Y,a}$ and
$\mathfrak m_x\subset\mathcal O_{X,x}$ denote the corresponding maximal
ideals.  We use
\begin{equation}\label{eq:pullback-ideal-order}
 \ord_x(F^*\mathfrak m_a)
 :=\max\bigl\{q\in\mathbf Z_{\geq0}:
 F^*\mathfrak m_a\cdot\mathcal O_{X,x}\subseteq\mathfrak m_x^q\bigr\}.
\end{equation}
For a finite holomorphic germ
$f=(f_1,\ldots,f_n):(\C^n,0)\to(\C^n,0)$, the corresponding definitions
are
\[
 e_0(f)=\dim_\C
 \mathcal O_{\C^n,0}/(f_1,\ldots,f_n),
 \qquad
 \ord_0(f^*\mathfrak m_0)=\min_i\ord_0f_i.
\]

\par\medskip
\begin{lem}[Generic target rotation]\label{lem:generic-rotation}
Let $f:(\C^n,0)\to(\C^n,0)$ be a finite holomorphic germ, and set
\[
 r=\ord_0(f^*\mathfrak m_0).
\]
For every psh germ $u$ that is locally bounded away from the origin, one
has
\[
 \nu(u\circ U\circ f,0)=r\nu(u,0)
\]
for Haar almost every $U\in\mathrm U(n)$.
\end{lem}

\begin{proof}
Write $f=f_r+O(\norm{z}_2^{r+1})$, where $f_r$ is the vector of
degree-$r$ homogeneous terms and is not identically zero.  Choose a unit
vector $v$ such that $q=f_r(v)\ne0$.  Then
\[
 \gamma(t):=f(tv)=t^rq+O(t^{r+1}).
\]
After shrinking representatives, finiteness gives $f^{-1}(0)=\{0\}$.
Thus $t\mapsto u(Uf(tv))$ is not identically $-\infty$ for any unitary
$U$.
After adding a constant and shrinking the target ball, assume that
$u\leq0$.  For $U\in\mathrm U(n)$ and $0<\tau\ll1$, put
\[
 b_U(\tau)=\frac{1}{\log\tau}\frac{1}{2\pi}
 \int_0^{2\pi}u\bigl(U\gamma(\tau e^{\mathrm i\theta})\bigr)
 \,d\theta.
\]
These quantities are nonnegative.  Their limit as $\tau\downarrow0$ is
the Lelong number $\ell_U$ of the one-variable subharmonic germ
$t\mapsto u(Uf(tv))$.

Let $A_u(s)$ be the spherical mean of $u$ on the Euclidean sphere of
radius $s$.
With normalized Haar measure, invariance gives
\[
 \int_{\mathrm U(n)}b_U(\tau)\,dU
 =\frac{1}{2\pi\log\tau}\int_0^{2\pi}
 A_u\bigl(\norm{\gamma(\tau e^{\mathrm i\theta})}_2\bigr)\,d\theta.
\]
Uniformly in $\theta$, the radii
$s_{\tau,\theta}=\norm{\gamma(\tau e^{\mathrm i\theta})}_2$ satisfy,
for constants $0<C_-\leq C_+$ independent of $\tau$ and $\theta$,
\[
 C_-\tau^r\leq s_{\tau,\theta}\leq C_+\tau^r,
 \qquad
 \frac{\log s_{\tau,\theta}}{\log\tau}\longrightarrow r.
\]
For these radii,
\[
 \frac{A_u(s_{\tau,\theta})}{\log\tau}
 =\frac{A_u(s_{\tau,\theta})}{\log s_{\tau,\theta}}
  \frac{\log s_{\tau,\theta}}{\log\tau}.
\]
The spherical-mean characterization makes the first factor tend to
$\nu(u,0)$ uniformly once all radii are sufficiently small, while the
second tends to $r$ uniformly in $\theta$.  Hence the right-hand side
tends to $r\nu(u,0)$.
Fatou's lemma now yields
\[
 \int_{\mathrm U(n)}\ell_U\,dU\leq r\nu(u,0).
\]
If $\nu(u,0)>0$, fix $0<\eta<\nu(u,0)$.  The logarithmic growth
characterization and a bound
$\norm{f(z)}_2\leq C_f\norm{z}_2^r$ for some $C_f>0$ give, for every
$U$,
\[
 u(Uf(z))\leq r\eta\log\norm{z}_2+O(1).
\]
Division by $\log\norm{z}_2<0$ reverses the inequality.  Letting
$\eta\uparrow\nu(u,0)$ yields
$\nu(u\circ U\circ f,0)\geq r\nu(u,0)$; when $\nu(u,0)=0$, the same
bound is simply nonnegativity.  Since restriction to $\C v$ can only
increase the Lelong number,
\[
 r\nu(u,0)\leq\nu(u\circ U\circ f,0)\leq\ell_U
 \qquad\text{for every }U.
\]
Together with the integral upper bound, this forces
$\ell_U=r\nu(u,0)$ almost everywhere, and the middle term has the same
value.
\end{proof}

\par\medskip
Transformation laws for pluricomplex Green functions with weighted poles
under proper holomorphic maps between domains were established by
Edigarian and Zwonek \cite{EZ98}.  The next lemma is a Bedford--Taylor
functoriality statement tailored to the compact global $DMA$ setting used
below.

\begin{lem}[Finite pullback of Bedford--Taylor measures]
\label{lem:finite-BT-pullback}
Let $F:X\to Y$ be a finite surjective holomorphic map of degree $d$
between connected complex manifolds of the same dimension $n$.  For a
positive Radon measure $\mu$ on $Y$, define $F^*\mu$ by requiring, for every
$\chi\in C_c^0(X)$,
\begin{equation}\label{eq:measure-pullback}
 \int_X\chi\,d(F^*\mu)
 =\int_Y\operatorname{Tr}_F\chi\,d\mu,
 \qquad
 \operatorname{Tr}_F\chi(y)
 =\sum_{z\in F^{-1}(y)}e_z(F)\chi(z).
\end{equation}
Then $\operatorname{Tr}_F\chi$ is continuous for every
$\chi\in C_c^0(X)$.  If $T$ is a positive closed $(1,1)$-current on
$Y$ with bounded local potentials, write $T=\ddc\phi$ locally and set
\[
 F^*T:=\ddc(\phi\circ F).
\]
This is independent of the chosen local potential.  Then
\begin{equation}\label{eq:bounded-finite-functoriality}
 (F^*T)^n=F^*(T^n).
\end{equation}
Consequently, if $(\phi_j)$ is a sequence of bounded psh functions on a
coordinate ball $B\subset Y$ that decreases to a possibly unbounded psh
function $\phi$, and if
\[
 (\ddc\phi_j)^n\rightharpoonup\mu
 \qquad\text{locally on }B,
\]
then
\[
 \bigl(\ddc(\phi_j\circ F)\bigr)^n
 \rightharpoonup F^*\mu
\]
locally on $F^{-1}(B)$.
\end{lem}

\begin{proof}
For $z\in X$, the finite local homomorphism
\[
 \mathcal O_{Y,F(z)}\longrightarrow\mathcal O_{X,z}
\]
has regular source and regular, hence Cohen--Macaulay, target.  Moreover,
\[
 \dim\mathcal O_{X,z}
 =\dim\mathcal O_{Y,F(z)}
  +\dim\bigl(\mathcal O_{X,z}/
      \mathfrak m_{F(z)}\mathcal O_{X,z}\bigr)
 =n+0.
\]
Miracle flatness \cite[Theorem~23.1]{Mat86} therefore makes $F$ flat.
Thus $F$ is finite locally free of rank $d$, every scheme-theoretic
fiber has length $d$, and its local lengths are the multiplicities
$e_z(F)$.

Fix $y_0\in Y$ and write $F^{-1}(y_0)=\{z_1,\ldots,z_s\}$.  Choose
pairwise disjoint relatively compact neighborhoods $W_i$ of $z_i$ on
which $\chi$ differs from $\chi(z_i)$ by less than a prescribed
$\epsilon>0$.  After shrinking to a connected neighborhood $B$ of
$y_0$, properness gives a disjoint decomposition
\[
 F^{-1}(B)=V_1\sqcup\cdots\sqcup V_s,
 \qquad z_i\in V_i\Subset W_i.
\]
Each $V_i\to B$ is finite locally free.  Its rank is constant and,
on evaluating the fiber over $y_0$, equals $e_{z_i}(F)$.  Hence
\[
 \sum_{z\in F^{-1}(y)\cap V_i}e_z(F)=e_{z_i}(F),
 \qquad y\in B.
\]
This constancy is the finite-type condition used in the trace theorem of
\cite[Section~5.16]{AGV85}; the elementary estimate below gives the
continuity needed here for arbitrary $\chi\in C_c^0(X)$.
It follows that
$|\operatorname{Tr}_F\chi(y)-\operatorname{Tr}_F\chi(y_0)|\leq d\epsilon$
for $y\in B$, proving continuity.  Moreover,
\[
 \operatorname{supp}(\operatorname{Tr}_F\chi)
 \subset F(\operatorname{supp}\chi),
\]
which is compact because $F$ is proper.
Also
\[
 |\operatorname{Tr}_F\chi|\leq d\norm{\chi}_\infty.
\]
Consequently,
$\chi\mapsto\int_Y\operatorname{Tr}_F\chi\,d\mu$ is a positive linear
functional on $C_c^0(X)$ that is bounded on functions supported in each
fixed compact set.  The Riesz--Markov representation theorem therefore
defines a unique positive Radon measure $F^*\mu$ satisfying
\eqref{eq:measure-pullback}.

The identity \eqref{eq:bounded-finite-functoriality} is local on $Y$.
Write $T=\ddc\phi$ on a coordinate ball $B$, with $\phi$ bounded and
psh.  On a smaller ball $B'\Subset B$, standard convolution
regularization \cite[Chapter~III, Proposition~3.2]{Dem12} gives smooth
psh functions $\phi^{(\nu)}\downarrow\phi$.  For
$T^{(\nu)}=\ddc\phi^{(\nu)}$, the
ramification locus and branch locus are
\[
 R_F=\{z\in X:\operatorname{rank}dF_z<n\},
 \qquad B_F=F(R_F).
\]
Because $F$ is a finite dominant map between smooth complex manifolds in
characteristic zero, it is generically \'{e}tale.  Thus $R_F$ is a proper
analytic subset.  Remmert's theorem makes $B_F=F(R_F)$ analytic; since a
finite map preserves the dimension of analytic subsets,
\[
 \dim B_F=\dim R_F\leq n-1.
\]
Hence $B_F$ is a proper closed analytic subset.

The identity
$(F^*T^{(\nu)})^n=F^*((T^{(\nu)})^n)$ holds on
$F^{-1}(B')\setminus R_F$ by the ordinary smooth pullback formula.  The
left-hand side has a smooth density and hence gives no mass to $R_F$.
The right-hand side can charge $R_F$ only over $B_F$, which has zero mass
for the smooth measure $(T^{(\nu)})^n$.  Thus the identity holds on all
of $F^{-1}(B')$.

Bedford--Taylor monotone continuity gives convergence on the left.  On
the right, for $\chi\in C_c^0(F^{-1}(B'))$, the trace formula gives
\[
 \int_{F^{-1}(B')}\chi\,dF^*((T^{(\nu)})^n)
 =\int_{B'}\operatorname{Tr}_F\chi\,d(T^{(\nu)})^n
 \longrightarrow
 \int_{B'}\operatorname{Tr}_F\chi\,dT^n.
\]
This proves \eqref{eq:bounded-finite-functoriality}.  Applying it to each
bounded $\phi_j$ in the final assertion and testing with
$\chi\in C_c^0(F^{-1}(B))$ gives
\[
 \int_{F^{-1}(B)}\chi\,\bigl(\ddc(\phi_j\circ F)\bigr)^n
 \longrightarrow
 \int_B\operatorname{Tr}_F\chi\,d\mu
 =\int_{F^{-1}(B)}\chi\,d(F^*\mu).
\]
This is the stated local convergence.  Any atom over a branch value is
already counted by the local multiplicities in $F^*\mu$; no additional
ramification term occurs.
\end{proof}

\par\medskip
\begin{prop}[Finite-pullback criterion]\label{prop:finite-pullback}
Let $n\geq2$, let $a=[1:0:\cdots:0]\in\Pj^n$, and let
$F:X\to\Pj^n$ be a finite surjective morphism of degree $d$ from a
smooth projective $n$-fold.  Put $c=d^{-1/n}$, let $\vartheta$ be a
K\"ahler form in the class $cF^*\{\omega_{\FS}\}$, and choose
$h\in C^\infty(X)$ such that
\[
 \vartheta=cF^*\omega_{\FS}+\ddc h.
\]
For $\lambda\in[0,1]$, set
\[
 U_\lambda=cG_{n,\lambda}\circ F-h.
\]
Then $U_\lambda\in\PSH(X,\vartheta)\cap DMA(X,\vartheta)$, its pole set
is $F^{-1}(a)$, and
\[
 (\vartheta+\ddc U_\lambda)^n
 =\frac1d\sum_{z\in F^{-1}(a)}e_z(F)\,\delta_z.
\]

Suppose in addition that $F^{-1}(a)=\{x\}$ set-theoretically, and put
\[
 r=\ord_x(F^*\mathfrak m_a).
\]
For every $\lambda\in[0,1]$, there is a projective unitary transformation
$R_\lambda$ fixing $a$ such that
\[
 W_\lambda=cG_{n,\lambda}\circ R_\lambda\circ F-h
\]
is a one-pole Green function at $x$ with
\[
 \nu(W_\lambda,x)=cr\lambda.
\]
Consequently,
\[
 [0,cr]\subseteq\mathcal R_{cF^*\{\omega_{\FS}\}}(x),
 \qquad
 \varepsilon(cF^*\{\omega_{\FS}\},x)\geq cr.
\]
If an irreducible curve $C\ni x$ satisfies
\[
 F^*\mathcal O_{\Pj^n}(1)\cdot C
 =r\,\operatorname{mult}_xC,
\]
then
\[
 \mathcal R_{cF^*\{\omega_{\FS}\}}(x)
 =[0,cr]
 =[0,\varepsilon(cF^*\{\omega_{\FS}\},x)].
\]
\end{prop}

\begin{proof}[Proof of \Cref{prop:finite-pullback}]
The class $cF^*\{\omega_{\FS}\}$ is K\"ahler because a finite pullback
of an ample line bundle is ample, and its volume is $c^nd=1$.
Moreover,
\[
 \vartheta+\ddc U_\lambda
 =cF^*(\omega_{\FS}+\ddc G_{n,\lambda})\geq0,
\]
so $U_\lambda$ is $\vartheta$-psh.
Let $G_{n,\lambda}^{(\ell)}\downarrow G_{n,\lambda}$ be bounded
$\omega_{\FS}$-psh functions and set
$U_\lambda^{(\ell)}=cG_{n,\lambda}^{(\ell)}\circ F-h$.  By
\Cref{lem:finite-BT-pullback},
\[
 (\vartheta+\ddc U_\lambda^{(\ell)})^n
 =\frac1dF^*\bigl((\omega_{\FS}+\ddc G_{n,\lambda}^{(\ell)})^n\bigr).
\]
Set
\[
 \mu_\ell=(\omega_{\FS}+\ddc G_{n,\lambda}^{(\ell)})^n.
\]
Since $G_{n,\lambda}\in DMA(\Pj^n,\omega_{\FS})$,
$\mu_\ell\rightharpoonup\delta_a$.  Hence, for every $\chi\in C^0(X)$,
\[
 \int_X\chi\,(\vartheta+\ddc U_\lambda^{(\ell)})^n
 =\frac1d\int_{\Pj^n}\operatorname{Tr}_F\chi\,d\mu_\ell
 \longrightarrow
 \frac1d\sum_{z\in F^{-1}(a)}e_z(F)\chi(z).
\]
Equivalently,
\begin{equation}\label{eq:finite-pullback-measure}
 (\vartheta+\ddc U_\lambda^{(\ell)})^n\rightharpoonup
 \frac1d\sum_{z\in F^{-1}(a)}e_z(F)\delta_z.
\end{equation}
The potential $U_\lambda$ is locally bounded away from the finite fiber
$F^{-1}(a)$.  By \Cref{lem:isolated-DMA}, it belongs to
$DMA(X,\vartheta)$, so \eqref{eq:finite-pullback-measure} is its
classical Monge--Amp\`ere measure.  Its pole set is $F^{-1}(a)$, and
\[
 \exp(U_\lambda)=e^{-h}
 \bigl(\exp(G_{n,\lambda})\circ F\bigr)^c
\]
extends continuously by zero along that fiber.

The same argument applies with $G_{n,\lambda}$ replaced by
$G_{n,\lambda}\circ R$ whenever $R$ is projective unitary and fixes
$a$.

Assume $F^{-1}(a)=\{x\}$, set
$\mathscr L=F^*\mathcal O_{\Pj^n}(1)$, and let
$r=\ord_x(F^*\mathfrak m_a)$.  For every irreducible curve $C\ni x$,
choose a hyperplane $H\ni a$ that does not contain $F(C)$.  A local
equation of $F^*H$ lies in $\mathfrak m_x^r$, and hence
\[
 \mathscr L\cdot C\geq i_x(F^*H,C)
 \geq r\,\operatorname{mult}_xC.
\]
For the second inequality, pull the local equation back to the
normalization of $C$: every element of $\mathfrak m_x^r$ then vanishes
to order at least $r\operatorname{mult}_xC$.  The curve characterization
of the Seshadri constant \cite[Proposition~5.1.5]{Laz04} gives
\begin{equation}\label{eq:seshadri-finite-lower}
 \varepsilon(c_1(\mathscr L),x)\geq r.
\end{equation}

Choose local coordinates $\zeta$ centered at $x$.  In the target affine
chart $[1:z]$ centered at $a$, put
\[
 \chi_{\FS}(z)=\frac12\log(1+\norm{z}_2^2),
 \qquad
 g_\lambda(z)=G_{n,\lambda}([1:z])+\chi_{\FS}(z).
\]
For $U\in\mathrm U(n)$, let $R_U$ be the projective unitary
transformation induced by $\operatorname{diag}(1,U)\in\mathrm U(n+1)$.
Then $R_U$ fixes $a$, induces $z\mapsto Uz$, and satisfies
\[
 \chi_{\FS}(Uz)=\chi_{\FS}(z),
 \qquad R_U^*\omega_{\FS}=\omega_{\FS}.
\]
Let $f(\zeta)$ be the target affine-coordinate expression of the germ of
$F$ in these source coordinates.  Then
\[
 g_\lambda\circ U\circ f
 =G_{n,\lambda}\circ R_U\circ F+\chi_{\FS}\circ f.
\]
The last term is smooth.  The germ $g_\lambda$ is psh, locally bounded
away from the origin, and has Lelong number $\lambda$.  Applying
\Cref{lem:generic-rotation} to $g_\lambda$ and the finite germ $f$, we
may therefore choose a projective unitary transformation $R_\lambda$
fixing $a$ so that
\[
 \nu(G_{n,\lambda}\circ R_\lambda\circ F,x)=r\lambda.
\]
Because $F$ is finite locally free of rank $d$, the fiber supported at
the single point $x$ has length $e_x(F)=d$.  The preceding pullback
calculation therefore gives
\[
 (\vartheta+\ddc W_\lambda)^n=\delta_x,
 \qquad
 \nu(W_\lambda,x)=cr\lambda.
\]
Thus $[0,cr]$ is realized.  By homogeneity and
\eqref{eq:seshadri-finite-lower},
\[
 \varepsilon(cF^*\{\omega_{\FS}\},x)
 =c\,\varepsilon(F^*\{\omega_{\FS}\},x)\geq cr.
\]
If the stated curve exists, its ratio and
\eqref{eq:seshadri-finite-lower} give
$\varepsilon(F^*\{\omega_{\FS}\},x)=r$.  The scaled Seshadri constant is
therefore $cr$, and the universal upper bound in
\Cref{prop:seshadri-obstruction} shows that the realized interval is the
complete range.
\end{proof}

\begin{remark}[Why the target rotation is needed]
The projective unitary transformation in \Cref{prop:finite-pullback} may
depend on $\lambda$.  A singleton fiber alone does not give a composition
formula for an anisotropic singularity.  For example, let
\[
 f(z_1,z_2)=(z_1^2,z_2^3),
 \qquad
 u(w_1,w_2)=\max\{2\log|w_1|,\log|w_2|\}.
\]
Then $f^{-1}(0)=\{0\}$,
$\ord_0(f^*\mathfrak m_0)=2$, and $\nu(u,0)=1$, whereas
$\nu(u\circ f,0)=3$.  The averaging in \Cref{lem:generic-rotation}
recovers the minimal-order factor for almost every target rotation.
\end{remark}

\subsection{Products of projective spaces}\label{subsec:products}

We now combine the finite-pullback criterion with an explicit finite
morphism to obtain the multiprojective theorem.

\begin{proof}[Proof of \Cref{thm:projective-space-products}]
Let $A^N$ denote the top self-intersection number and set
\[
 \mathcal V:=A^N
 =\frac{N!}{\prod_{j=1}^k n_j!}\prod_{j=1}^k a_j^{n_j}.
\]
Choose an integer $m_0\geq1$ such that $b_j:=m_0a_j$ is a positive
integer for every $j$, and first take
\[
 x=([1:0:\cdots:0],\ldots,[1:0:\cdots:0]).
\]
Fix independent binary linear forms $S,T$ and identify $\Pj^{n_j}$ with
the projective space of binary forms of degree $n_j$ by using the ordered
monomial basis
\[
 S^{n_j},S^{n_j-1}T,\ldots,T^{n_j}.
\]
Thus the coefficient coordinates are $[Z_0:\cdots:Z_{n_j}]$ and the
coordinate point is $[S^{n_j}]$.  We make the analogous identification
of the target $\Pj^N$, so that $[S^N]$ is its coordinate point
$[1:0:\cdots:0]$.

Multiplication of binary forms defines
\[
 \mathcal M:\prod_{j=1}^k\Pj^{n_j}\longrightarrow\Pj^N,
 \qquad ([\mathcal B_1],\ldots,[\mathcal B_k])
 \longmapsto[\mathcal B_1\cdots\mathcal B_k].
\]
To see that $\mathcal M$ is finite, write a fixed binary form as
\[
 \mathcal B=\prod_{\xi\in\Pj^1}L_\xi^{m_\xi},
\]
where $L_\xi$ is a binary linear form vanishing at $\xi$.
A factorization $\mathcal B=\mathcal B_1\cdots\mathcal B_k$ with
$\deg\mathcal B_j=n_j$ is determined by nonnegative integers
$m_{\xi,j}$ satisfying
\[
 \sum_jm_{\xi,j}=m_\xi,
 \qquad \sum_\xi m_{\xi,j}=n_j.
\]
Only finitely many arrays satisfy these conditions, so every fiber is
finite.  The morphism is proper because its source is projective; hence
it is finite.  It is surjective because every binary form splits into
linear factors over $\C$.  Multihomogeneity gives
\begin{equation}\label{eq:binary-multiplication-data}
 \mathcal M^*\mathcal O_{\Pj^N}(1)=\mathcal O_X(1,\ldots,1),
 \qquad
 \deg\mathcal M=\frac{N!}{\prod_jn_j!}.
\end{equation}
Indeed, since $\mathcal M$ is finite and surjective, the projection
formula and the normalization of the hyperplane class give
\[
 \deg\mathcal M
 =\int_X(H_1+\cdots+H_k)^N
 =\frac{N!}{\prod_jn_j!}.
\]
Here $H_j^{n_j+1}=0$ and
$\int_X\prod_jH_j^{n_j}=1$, so only the displayed top-degree monomial
contributes.
The only factorization of $S^N$ with these prescribed degrees is
\[
 \mathcal M^{-1}([S^N])
 =\{([S^{n_1}],\ldots,[S^{n_k}])\}.
\]

On the $j$-th factor define the coordinate power morphism
\[
 \varpi_j[Z_0:\cdots:Z_{n_j}]
 =[Z_0^{b_j}:\cdots:Z_{n_j}^{b_j}],
\]
and put $F=\mathcal M\circ\prod_j\varpi_j$.  Each $\varpi_j$ is finite
and satisfies $\varpi_j^*\mathcal O_{\Pj^{n_j}}(1)
=\mathcal O_{\Pj^{n_j}}(b_j)$; hence $\deg\varpi_j=b_j^{n_j}$.
It also has the coordinate
point as the unique inverse image of $[S^{n_j}]$.  It follows from
\eqref{eq:binary-multiplication-data} that
\begin{equation}\label{eq:product-map-data}
 \begin{aligned}
 F^*\mathcal O_{\Pj^N}(1)
   &=\mathcal O_X(b_1,\ldots,b_k),
 &c_1\bigl(\mathcal O_X(b_1,\ldots,b_k)\bigr)&=m_0A,\\
 \deg F
   &=\frac{N!}{\prod_jn_j!}\prod_jb_j^{n_j}
     =m_0^N\mathcal V.
 \end{aligned}
\end{equation}
Moreover, $F^{-1}([S^N])=\{x\}$ set-theoretically.

It remains to compute the order of this singleton fiber.  Write a binary
form near $[S^{n_j}]$ as
\[
 S^{n_j}+\sum_{\ell=1}^{n_j}z_{j,\ell}S^{n_j-\ell}T^\ell.
\]
After applying $\varpi_j$, the coefficient $z_{j,\ell}$ is replaced by
$z_{j,\ell}^{b_j}$.  Thus, for $1\leq\ell\leq N$, the nonconstant
coefficients of the product are
\[
 C_\ell=
 \sum_{\substack{\ell_1+\cdots+\ell_k=\ell\\0\leq\ell_j\leq n_j}}
 \prod_{j=1}^kz_{j,\ell_j}^{b_j},
 \qquad z_{j,0}=1.
\]
Every monomial occurring here has ordinary total degree at least
$b_{\min}:=\min_jb_j$.  In
$C_1=\sum_jz_{j,1}^{b_j}$, the lowest-degree terms involve distinct
variables and cannot cancel, so $C_1$ has order exactly $b_{\min}$.
In this affine chart the ideal $F^*\mathfrak m_{[S^N]}$ is generated by
$C_1,\ldots,C_N$; its order is therefore the minimum of their vanishing
orders.  Hence
\begin{equation}\label{eq:product-fiber-order}
 \ord_x(F^*\mathfrak m_{[S^N]})
 =b_{\min}=m_0\min_ja_j.
\end{equation}

The normalized pullback class in \Cref{prop:finite-pullback} is
\[
 (\deg F)^{-1/N}F^*c_1(\mathcal O_{\Pj^N}(1))
 =\frac{A}{\mathcal V^{1/N}}=\alpha.
\]
Choose an index $j_0$ with $b_{j_0}=b_{\min}$, take a line through the
coordinate point in the $j_0$-th factor, and keep all other factors fixed.
The resulting curve $\ell\subset X$ satisfies
\[
 F^*\mathcal O_{\Pj^N}(1)\cdot\ell=b_{\min}
 =\ord_x(F^*\mathfrak m_{[S^N]})\operatorname{mult}_x\ell,
 \qquad \operatorname{mult}_x\ell=1.
\]
The sharpness part of \Cref{prop:finite-pullback} now gives
\[
 \mathcal R_\alpha(x)
 =[0,\varepsilon(\alpha,x)]
 =\left[0,\frac{\min_ja_j}{\mathcal V^{1/N}}\right].
\]
Finally, let $x_0$ be the coordinate point used above and choose
$g\in\prod_j\operatorname{PGL}(n_j+1,\C)$ with $g(x_0)=x$.  Since
$g^*H_j=H_j$, one has $g^*\alpha=\alpha$.  If a positive current $T$ in
$\alpha$ realizes a value $t$ at $x_0$, then $(g^{-1})^*T$ lies in the
same class, has top Monge--Amp\`ere measure $\delta_x$, and has Lelong
number $t$ at $x$.  Replacing its smooth representative by the chosen
representative of $\alpha$, as in \eqref{eq:projective-range}, changes
the potential only by a smooth function.  Thus
$\mathcal R_\alpha(x)=\mathcal R_\alpha(x_0)$, and the same formula holds
at every point.
\end{proof}

\subsection{A further consequence}\label{subsec:further-consequence}

\begin{cor}[Equal-weight multipoles]\label{cor:projective-multipoles}
Let $X$ be a smooth projective $n$-fold, $n\geq2$, and let $\mathscr L$ be an
ample line bundle.  Put
\[
 \alpha_{\mathscr L}=\frac{c_1(\mathscr L)}{(c_1(\mathscr L)^n)^{1/n}}.
\]
For every K\"ahler form $\vartheta_{\mathscr L}$ with
$\{\vartheta_{\mathscr L}\}=\alpha_{\mathscr L}$,
there are an integer
$d\geq1$ and distinct points $x_1,\ldots,x_d\in X$ such that, for every
$\lambda\in[0,1]$, there is
$U_\lambda\in\PSH(X,\vartheta_{\mathscr L})\cap
DMA(X,\vartheta_{\mathscr L})$ with
\[
 (\vartheta_{\mathscr L}+\ddc U_\lambda)^n
 =\frac1d\sum_{i=1}^d\delta_{x_i},
 \qquad
 \nu(U_\lambda,x_i)=d^{-1/n}\lambda.
\]
\end{cor}

\begin{proof}
Choose an integer $m_1\geq1$ such that $\mathscr L^{\otimes m_1}$ is
very ample and embed $X$ by $|\mathscr L^{\otimes m_1}|$.
Starting from $X\subset\Pj^M$, restrict successive projective linear
projections $\Pj^q\dashrightarrow\Pj^{q-1}$ whose centers lie outside
the current image, until the target is $\Pj^n$.  These are unrelated to
the factor projections $\operatorname{pr}_j$ in
\Cref{thm:projective-space-products}.  At each step, every fiber is
contained in a line through the projection center.  A
positive-dimensional fiber would be the whole line and would therefore
contain the center, contradicting its choice.  Thus each restricted
projection is proper and quasi-finite, hence finite.  At every step,
hyperplanes in the new target pull back to hyperplanes through the
projection center.  Since that center is disjoint from the current
image, this linear subsystem is base-point free on the image; hence the
restricted projection preserves the pullback of the hyperplane bundle.
Induction from
$\mathcal O_{\Pj^M}(1)|_X\simeq\mathscr L^{\otimes m_1}$ shows that the
composition $F:X\to\Pj^n$ satisfies
\[
 F^*\mathcal O_{\Pj^n}(1)\simeq\mathscr L^{\otimes m_1}.
\]
Its image is a closed $n$-dimensional subvariety of $\Pj^n$, hence all
of $\Pj^n$.

In characteristic zero, the finite dominant map $F$ is generically
\'{e}tale.  Its ramification locus is a proper closed subset, and the
branch locus, namely its image under the finite map $F$, is a proper
closed subset of $\Pj^n$.
Choose a value outside the branch locus and
postcompose $F$ with a projective unitary
transformation so that this value is the point $a$ used in
\Cref{thm:main}.  Its fiber consists of
$d=\deg F=m_1^nc_1(\mathscr L)^n$ distinct points.  Apply
\Cref{prop:finite-pullback}; the normalized pullback class is
\[
 d^{-1/n}F^*c_1(\mathcal O_{\Pj^n}(1))=\alpha_{\mathscr L},
\]
and $F$ is a local biholomorphism at each point $x_i$ of the fiber.
Local biholomorphisms preserve Lelong numbers, while the correction $h$
is smooth.  Hence
$\nu(U_\lambda,x_i)=d^{-1/n}\nu(G_{n,\lambda},a)
=d^{-1/n}\lambda$.
\end{proof}

\section*{Acknowledgments}

The author used ChatGPT (OpenAI) during the preparation of the manuscript
for exploratory drafting and proof-audit assistance.  All mathematical arguments,
references, and conclusions were independently verified by the author, who
assumes full responsibility for the contents of the paper.

\end{document}